\documentclass[a4paper,11pt]{article} 

\usepackage{graphicx}
\usepackage{amssymb,amsmath,amsthm,mathrsfs,xfrac}
\usepackage{enumitem}
\usepackage{a4wide}
\usepackage{hyperref}
\usepackage{cleveref}
\usepackage{array}
\usepackage{xcolor}

\numberwithin{equation}{section}
\numberwithin{figure}{section}
\allowdisplaybreaks[4]

\newtheorem{proposition}{Proposition}[section]
\newtheorem{theorem}[proposition]{Theorem}
\newtheorem{lemma}[proposition]{Lemma}
\newtheorem{definition}[proposition]{Definition}

\newtheorem{remark}[proposition]{Remark}

\renewenvironment{proof}{\smallskip\noindent\emph{\textbf{Proof.}}%
  \hspace{1pt}}{\hspace{-5pt}{\nobreak\quad\nobreak\hfill\nobreak%
    $\square$\vspace{2pt}\par}\smallskip\goodbreak}

\newenvironment{proofof}[1]{\smallskip\noindent{\textbf{Proof~of~#1.}}%
  \hspace{1pt}}{\hspace{-5pt}{\nobreak\quad\nobreak\hfill\nobreak%
    $\square$\vspace{2pt}\par}\smallskip\goodbreak}

\newcommand{\Id}{\mathinner{\mathrm{Id}}}

\newcommand{\C}[1]{\mathbf{C}^{#1}}
\newcommand{\Cc}[1]{\mathbf{C}_c^{#1}}

\newcommand{\BV}{\mathbf{BV}}

\renewcommand{\L}[1]{\mathbf{L}^#1}
\newcommand{\Lloc}[1]{{\mathbf{L}_{\mathbf{loc}}^{#1}}}

\newcommand{\W}[2]{{\mathbf{W}^{#1,#2}}}

\newcommand{\modulo}[1]{{\left|#1\right|}}
\newcommand{\norma}[1]{{\left\|#1\right\|}}
\newcommand{\norm}[1]{{\left\|#1\right\|}}
\newcommand{\caratt}[1]{{\chi_{\strut#1}}}
\newcommand{\reali}{{\mathbb{R}}}
\newcommand{\R}{\mathbb{R}}
\newcommand{\naturali}{{\mathbb{N}}}
\newcommand{\N}{\mathbb{N}}

\renewcommand{\epsilon}{\varepsilon}
\renewcommand{\phi}{\varphi}
\renewcommand{\theta}{\vartheta}
\renewcommand{\O}{\mathcal{O}(1)}
\newcommand{\tv}{\mathinner{\rm TV}}
\newcommand{\spt}{\mathop{\rm spt}}

\renewcommand{\d}[1]{\mathinner{\mathrm{d}{#1}}}
\newcommand{\dd}[1]{\mathinner{\mathrm{d}{#1}}}
\renewcommand{\div}{\mathinner{\nabla\cdot}} 
\newcommand{\grad}{\mathinner{\nabla}}

\newcommand{\sOmega}{\mathop{\smash[b]{\,*_{{\strut\Omega}\!}}}}

\makeatletter
\let\@fnsymbol\@arabic
\makeatother

\title{Non Linear Hyperbolic--Parabolic Systems \\ with Dirichlet Boundary Conditions}

\author{Rinaldo M.~Colombo\footnotemark[1] \and Elena Rossi\footnotemark[2]}
\date{}
\begin{document}
\maketitle
\footnotetext[1]{INdAM Unit \& Department of Information Engineering,
  University of Brescia, via Branze, 38, 25123 Brescia, Italy. Email:
  \texttt{rinaldo.colombo@unibs.it}} \footnotetext[2]{University of
  Modena and Reggio Emilia, INdAM Unit \& Department of Sciences and
  Methods for Engineering, Via Amendola 2 -- Pad.~Morselli, 42122
  Reggio Emilia, Italy. Email: \texttt{elena.rossi13@unimore.it}}

\begin{abstract}
  \noindent We prove the well posedness of a class of non linear and
  non local mixed hyperbolic--parabolic systems in bounded domains,
  with Dirichlet boundary conditions. In view of control problems,
  stability estimates on the dependence of solutions on data and
  parameters are also provided. These equations appear in models
  devoted to population dynamics or to epidemiology, for instance.

  \medskip

  \noindent\textit{2020~Mathematics Subject Classification:} 35M30,
  35L04, 35K20

  \medskip


\end{abstract}


\section{Introduction}
\label{sec:Intro}

We consider the following non linear system on a bounded domain
$\Omega \subset \reali^n$
\begin{equation}
  \label{eq:1}
  \left\{
    \begin{array}{l}
      \partial_t u
      + \div \left(u \,v(t,w) \right)
      =
      \alpha (t,x,w) \, u + a(t,x) \,,
      \\
      \partial_t w
      - \mu \, \Delta w
      =
      \beta (t,x,u,w) \, w + b (t,x)\,,
    \end{array}
  \right.
  \quad (t,x) \in [0,T] \times \Omega \,.
\end{equation}
Systems of this form arise, for instance, in predator--prey
systems~\cite{parahyp} and can be used in the control of parasites,
see~\cite{SAPM2021, Pfab2018489}. A similar mixed
hyperbolic--parabolic system is considered, in one space dimension,
in~\cite{MR4273477}, where Euler equations substitute the balance law
in~\eqref{eq:1}.

Motivated by these applications, terms in~\eqref{eq:1} may well
contain non local functions of the unknowns. Typically, whenever $u$
is a \emph{predator} and $w$ a \emph{prey}, the speed $v$ governing
the movement of $u$, when computed at a point $x$, i.e.,
$\left(v \left(t,w\right)\right) (x)$, depends on $w$ through
integrals of the form
$\int_{\norma{x-\xi} \leq \rho} f \left(t,x, \xi, w
  (t,\xi)\right)\d\xi$ so that $\rho$ is the \emph{horizon} at which
the predator feels the prey.

Under standard assumptions on the functions defining~\eqref{eq:1}, we
provide the analytical framework where the existence and the
uniqueness of solutions to~\eqref{eq:1}--\eqref{eq:49}.  Moreover, we
obtain a full set of \emph{a priori} and stability estimates on view
of the interest about control problems based on these equations,
see~\cite{MR4456851}. To this aim, we equip~\eqref{eq:1} with
homogeneous Dirichlet boundary conditions and initial data:
\begin{equation}
  \label{eq:49}
  \left\{
    \begin{array}{rcl}
      u (t,\xi)
      & =
      & 0
      \\
      w (t, \xi)
      & =
      & 0
    \end{array}
  \right.
  \quad(t,\xi) \in [0,T] \times \partial\Omega
  \qquad \mbox{ and } \qquad
  \left\{
    \begin{array}{rcl}
      u (0,x)
      & =
      & u_o (x)
      \\
      w (0,x)
      & =
      & w_o (x)
    \end{array}
  \right.
  \quad x \in \Omega \,.
\end{equation}

We stress that the whole construction is settled in $\L1$, a usual
choice for balance laws but less common in the case of the parabolic
equation. This choice is motivated by the clear physical meaning of
\emph{total population} attached to this norm, whenever solutions are
positive -- a standard situation in the motivating models. As is well
known, in parabolic equations, $\L2$ or $\W{k}{2}$ are more standard
choices, also thanks to the further properties of reflexive spaces,
see for instance the recent papers~\cite{Colli, HanSongYu}.

The introduction of a boundary, with the corresponding boundary
conditions, affects the whole analytical structure, differently in the
two equations. Indeed, as is well known, the hyperbolic equation for
$u$ may well lead to problems that are locally overdetermined,
resulting in the boundary condition to be simply neglected,
see~\cite{MR542510, siam2018, MR2322018, MR1884231}. On the contrary,
the solution to the parabolic equation attains along the boundary the
prescribed value, for all positive times, see~\cite{FriedmanBook,
  MR0241822, QuittnerSouplet}.

We stress that in the hyperbolic case, different definitions of
solutions are available, see~\cite{MR542510, MR2322018, MR3819847,
  MR1884231}. Here, we provide Definition~\ref{def:Hyp} that unifies
different approaches, also allowing to prove an intrinsic uniqueness
of solutions, i.e., independent of the way solutions are constructed.

Particularly relevant are the estimates on the dependence of $(u,w)$
on the terms $a,b$ in~\eqref{eq:1}, which typically play the role of
controls. Indeed, in the applications of~\eqref{eq:1} to biological
problems, $a$ and $b$ typically measure the deployment of parasitoids
or chemicals that hinder the propagation or reproduction of harmful
parasites, see~\cite{SAPM2021, Pfab2018489}. It is with reference to
this context that we care to ensure the positivity of solutions,
whenever the data and the controls are positive.

\medskip

The next section, after the necessary introduction of the notation,
presents the result. Proofs and further technical details are
deferred to Section~\ref{sec:proofs}, where different paragraphs refer
to the parabolic problem, to the hyperbolic one and to the coupling.

\section{Main Results}
\label{sec:main}

Throughout, the following notation is used.
$\R_+ = [0, +\infty\mathclose[$, $\R_- =\mathopen]-\infty,0]$. If
$A \subseteq \R^n$, the characteristic function $\caratt{A}$ is
defined by $\caratt{A} (x) = 1$ iff $x \in A$ and $\caratt{A} (x) = 0$
iff $x \in \R^n\setminus A$. For $x_o \in \R^n$ and $r>0$, $B (x_o,r)$
is the open sphere centered at $x_o$ with radius $r$.  We fix a time
$T > 0$ and the following condition on the spatial domain $\Omega$:
\begin{enumerate}[label=$\mathbf{(\Omega)}$]
\item\label{it:omega} $\Omega$ is a non empty, bounded and connected
  open subset of $\reali^n$, with $\C{2,\gamma}$ boundary, for a
  $\gamma \in \mathopen]0,1\mathclose]$.
\end{enumerate}

\noindent This condition is mainly motivated by the treatment of the
parabolic part. Here we mostly use the framework in~\cite[Appendix~B,
\S~48]{QuittnerSouplet}. Other possible regularity assumptions on
$\partial\Omega$ are in~\cite[Chapter~4, \S~4, p.~294]{MR0241822}.

\medskip

We pose the following assumptions on the functions appearing in
problem~\eqref{eq:1}:
\begin{enumerate}[label=$\mathbf{(\boldsymbol{v})}$]
\item\label{it:v}
  $v: [0,T] \times \L\infty (\Omega; \R) \to (\C2 \cap \W1\infty)
  (\Omega; \R^n)$ is such that for a constant $K_v>0$ and for a map
  $C_v \in \Lloc\infty ([0,T]\times\reali_+; \reali_+)$ non decreasing
  in each argument, for all $t,t_1,t_2 \in [0,T]$ and
  $w,w_1, w_2 \in \L\infty (\Omega;\reali)$,
  \begin{align*}
    \norma{v (t,w)}_{\L\infty (\Omega;\R^n)}
    \leq \
    & K_v \, \norma{w}_{\L1 (\Omega;\R)}
    \\
    \norma{D_x v (t,w)}_{\L\infty (\Omega;\R^n)}
    \leq \
    & K_v\, \norma{w}_{\L1 (\Omega;\R)}
    \\
    \norma{v (t_1,w_1) - v (t_2,w_2)}_{\L\infty (\Omega; \R^n)}
    \leq \
    & K_v \left(
      \modulo{t_2 - t_1} + \norma{w_2 - w_1}_{\L1 (\Omega; \R)}\right)
    \\
    \norma{D^2_x v (t,w)}_{\L1 (\Omega;\R^{n\times n})}
    \leq \
    & C_v \left(t,\norma{w}_{\L1 (\Omega;\R)}\right) \,
      \norma{w}_{\L1 (\Omega;\R)}
    \\
    \norma{\div \left(v (t_1,w_1) - v (t_2,w_2)\right)}_{\L\infty (\Omega; \R)}
    \leq \
    & C_v \left(t,\max_{i=1,2}\norma{w_i}_{\L1 (\Omega;\R)}\right) \,
      \norma{w_1 - w_2}_{\L1 (\Omega;\R)} \,.
  \end{align*}
\end{enumerate}
\begin{enumerate}[label=$\mathbf{(\boldsymbol{\alpha})}$]
\item\label{it:alpha} $\alpha: [0,T] \times \Omega \times \R \to \R$
  admits a constant $K_\alpha> 0$ such that, for a.e.~$t \in [0,T]$
  and all $ w, w_1,w_2 \in \R$
  \begin{align*}
    \sup_{x \in \Omega} \modulo{\alpha (t,x,w_1) - \alpha (t,x,w_2)}
    \le \
    & K_\alpha \; \modulo{w_1-w_2}
    \\
    \sup_{(x,w)\in\Omega\times\R} \alpha (t,x,w) \le \
    & K_\alpha \left(1+w\right)
  \end{align*}
  and for all $w \in\BV(\Omega; \R)$
  \begin{displaymath}
    \tv \left(\alpha\left(t,\cdot,w (t,\cdot)\right)\right)
    \le
    K_\alpha  \left(1+\norma{w}_{\L\infty (\Omega; \R)} + \tv (w)\right) \,.
  \end{displaymath}

\end{enumerate}
\begin{enumerate}[label=$\mathbf{(\boldsymbol{a})}$]
\item\label{it:a} $a \in \L1 \left([0,T]; \L\infty(\Omega; \R)\right)$
  and for all $t \in [0,T]$, $a (t) \in \BV (\Omega;\R)$.
\end{enumerate}
\begin{enumerate}[label=$\mathbf{(\boldsymbol{\beta})}$]
\item\label{it:betaFinal}
  $\beta: [0,T] \times \Omega \times \R \times R \to \R$ admits a
  constant $K_\beta > 0$ such that, for a.e.~$t \in [0,T]$ and all
  $u, u_1,u_2,w, w_1,w_2 \in \R$
  \begin{align*}
    \sup_{x \in \Omega} \modulo{\beta (t,x,u_1,w_1) - \beta (t,x,u_2,w_2)}
    \le \
    & K_\beta \left(\modulo{u_1-u_2} + \modulo{w_1-w_2}\right)
    \\
    \sup_{(x,u,w)\in\Omega\times\R \times \R} \beta (t,x,u,w) \le \
    & K_\beta \,.
  \end{align*}
\end{enumerate}
\begin{enumerate}[label=$\mathbf{(\boldsymbol{b})}$]
\item\label{it:b} $b \in \L1 ([0,T]; \L\infty( \Omega;\R))$ and for
  all $t \in [0,T]$, $b (t) \in \BV (\Omega;\R_+)$.
\end{enumerate}

\noindent Note in particular that~\ref{it:v} requires to bound
$\L\infty$ norms by means of $\L1$ norms, a feature typical of non
local operators. In fact, referring to predator--prey applications, it
is in general reasonable to assume that the $u$ (predator) population
moves according to \emph{averages} of the $w$ (prey) population
density or of its gradient. This justifies our requiring $v$
in~\eqref{eq:1} to be a \emph{non local} function of $w$.

Since we deal with the bounded domain $\Omega$, these averages need to
be computed only inside $\Omega$. To this aim, the modified
convolution introduced in~\cite[\S~3]{siam2018}, which reads
\begin{equation}
  \label{eq:58}
  (\rho \sOmega \eta) (x)
  =
  \dfrac{\int_\Omega \rho (y) \; \eta (x-y) \d{y}}{\int_{\Omega} \eta (x-y) \d{y}}
\end{equation}
is of help. The quantity $(\rho \sOmega \eta) (x)$ is an average of
the crowd density $\rho$ in $\Omega$ around $x$ as soon as the kernel
$\eta$ satisfies
\begin{description}
\item[($\boldsymbol{\eta}$)] $\eta (x) = \tilde\eta (\norma{x})$,
  where $\tilde\eta \in \C2 (\reali_+; \reali)$,
  $\spt \tilde\eta = [0, \ell_\eta]$, $\ell_\eta > 0$,
  $\tilde\eta' \leq 0$, $\tilde\eta' (0) = \tilde\eta'' (0) = 0$ and
  $\int_{\reali^N} \eta (\xi) \d\xi = 1$.
\end{description}

\noindent In those models where it is reasonable to assume that $u$
moves directed towards the areas with higher/lower density of $w$,
i.e., $v$ is parallel to the average gradient of $w$ in $\Omega$, we
select:
\begin{equation}
  \label{eq:61}
  v (t,w)
  \quad\big/\!\!\big/\quad
  \dfrac{\grad (w \sOmega \eta)}{\sqrt{1 + \norma{\grad (w \sOmega \eta)}^2}}
\end{equation}
where as kernel $\eta$ we choose for instance
$\eta (x) = \overline{\ell} \left(\ell^4 -
  \norma{x}^4\right)^4$. Here, $\ell$ has the clear physical meaning
of the distance, or \emph{horizon}, at which individuals of the $u$
population \emph{feel} the presence of the $w$ population.  The
normalization parameter $\overline\ell$ is chosen so that
$\int_{\R^2} \eta (x) \d{x} = 1$.  A choice like~\eqref{eq:61} is
consistent with the requirements~\ref{it:v}, as proved
in~\cite[Lemma~3.2]{siam2018}.

\medskip

To state what we mean by a \emph{solution} to~\eqref{eq:1}, we resort
to the standard definitions of solutions, separately, to the
hyperbolic and to the parabolic problems constituting~\eqref{eq:1}. In
the former case, we refer to~\cite{MR2322018, MR3819847, MR1884231}
and in the latter to the classical~\cite{QuittnerSouplet}.

\begin{definition}
  \label{def:sol}
  A pair $(u,w) \in \C0 \left([0,T]; \L1 (\Omega;\reali^2)\right)$ is
  a solution to~\eqref{eq:1}--~\eqref{eq:49} if, setting
  \begin{displaymath}
    \begin{array}{rcl}
      c (t,x)
      & =
      & v (t,w) (x)
      \\
      A (t,x)
      & =
      & \alpha \left(t,x,w (t,x)\right)
    \end{array}
    \qquad\qquad
    B (t,x)
    =
    \beta \left(t,x,u (t,x),w (t,x)\right) \,,
  \end{displaymath}
  the function $u$, according to Definition~\ref{def:Hyp}, solves
  \begin{displaymath}
    \left\{
      \begin{array}{l@{\qquad}r@{\,}c@{\,}l}
        \partial_t u
        + \div \left(u \, c(t,x) \right)
        =
        A (t,x)\, u + a(t,x)
        & (t,x)
        & \in
        & [0,T] \times \Omega
        \\
        u (t,\xi) = 0
        & (t,\xi)
        & \in
        & [0,T] \times \partial \Omega
        \\
        u (0,x) = u_o (x)
        & x
        & \in
        & \Omega
      \end{array}
    \right.
  \end{displaymath}
  and the function $w$, according to Definition~\ref{def:Para}, solves
  \begin{displaymath}
    \left\{
      \begin{array}{l@{\qquad}r@{\,}c@{\,}l}
        \partial_t w
        - \mu \, \Delta w
        =
        B (t,x) \, w + b (t,x)
        & (t,x)
        & \in
        & [0,T] \times \Omega
        \\
        w (t,\xi) = 0
        & (t,\xi)
        & \in
        & [0,T] \times \partial \Omega
        \\
        w (0,x) = w_o (x)
        & x
        & \in
        & \Omega \,.
      \end{array}
    \right.
  \end{displaymath}
\end{definition}

\noindent In the present framework, we also verify that, under
suitable conditions on the initial data, the solution $(u,w)$ enjoys
the following regularity
$\left(u (t), w (t)\right) \in (\BV \cap \L\infty) (\Omega;
\reali_+^2)$ for all $t \in [0,T]$.

We are now ready to state the main result of this work.

\begin{theorem}
  \label{thm:main}
  Let~\ref{it:omega}--\ref{it:v}--\ref{it:alpha}--\ref{it:a}--\ref{it:betaFinal}--\ref{it:b}
  hold. For any initial datum $(u_o,w_o)$ in
  $(\L\infty \cap \BV) (\Omega;\reali^2)$, problem~\eqref{eq:1} admits
  a unique solution on $[0,T]$ in the sense of
  Definition~\ref{def:sol}. Moreover, the following properties hold:

  \paragraph{A priori bounds:} There exists a constant $C$ depending
  only on $\Omega$, $K_\alpha,K_\beta,K_v$ such that for all
  $t \in [0,T]$ and for all initial data
  \begin{align}
    \label{eq:51}
    \norma{w (t)}_{\L1 (\Omega; \R)}
    \le \
    & e^{C \, t} \left(
      \norma{w_o}_{\L1 (\Omega;\reali)}
      +
      \norma{b}_{\L1 ([0,t]\times\Omega;\reali)}
      \right)
    \\
    \label{eq:52}
    \norma{w (t)}_{\L\infty (\Omega;\reali)}
    \le \
    & e^{C \, t}  \left(
      \norma{w_o}_{\L\infty (\Omega;\reali)}
      +
      \norma{b}_{\L1 ([0,t]; \L\infty(\Omega;\reali)}
      \right) \,.
  \end{align}
  Let $C_w (t)$ denote the maximum of the two right hand sides
  in~\eqref{eq:51} and in~\eqref{eq:52}; then,
  \begin{align}
    \label{eq:54}
    \norm{u (t)}_{\L1 (\Omega; \R)}
    \le \
    & \left(\norm{u_o}_{\L1 (\Omega; \R)}
      + \norm{a}_{\L1 ([0,t] \times \Omega; \R)}\right)
      \exp\left(C \, t
      \left(
      1
      +
      C_w (t)
      \right)
      \right)
    \\
    \label{eq:55}
    \norm{u (t)}_{\L\infty (\Omega; \R)}
    \le \
    & \left(\norm{u_o}_{\L\infty (\Omega; \R)}
      + \norm{a}_{\L1 ([0,t]; \L\infty (\Omega; \R))}\right)
      \exp\left(C \, t
      \left(
      1
      +
      2\, C_w (t)        \right)
      \right) \,.
  \end{align}

  \paragraph{Lipschitz continuity in the initial data:} Let
  $(\tilde u_o, \tilde w_o) \in (\L\infty \cap \BV) (\Omega;\reali^2)$
  and call $(\tilde u, \tilde w)$ the corresponding solution
  to~\eqref{eq:1}. Then, for all $t \in [0,T]$,
  \begin{displaymath}
    \norm{u(t)- \tilde u (t)}_{\L1 (\Omega; \R)}
    +
    \norm{w(t)- \tilde w (t)}_{\L1 (\Omega; \R)}
    \le \mathcal{C}(t)
    \left(
      \norm{u_o - \tilde u_o}_{\L1 (\Omega; \R)}
      + \norm{w_o - \tilde w_o}_{\L1 (\Omega; \R)}\right),
  \end{displaymath}
  where $\mathcal{C} \in \L\infty ([0,T];\R_+)$ depends on $\Omega$,
  $K_\alpha$, $K_\beta$, $K_v$, on the map $C_v$, on norms and total
  variation of the functions $a$ and $b$ and of the initial data.

  \paragraph{Stability with respect to the controls:} Let $\tilde a$
  satisfy~\ref{it:a}, $\tilde b$ satisfy~\ref{it:b} and call
  $(\tilde u, \tilde w)$ the corresponding solution
  to~\eqref{eq:1}. Then, for all $t \in [0,T]$,
  \begin{displaymath}
    \norm{u(t)- \tilde u (t)}_{\L1 (\Omega; \R)}
    +
    \norm{w(t)- \tilde w (t)}_{\L1 (\Omega; \R)}
    \le \mathcal{C}(t)\!
    \left(
      \norm{a - \tilde a}_{\L1 ([0,t] \times \Omega; \R)}
      \!+\!
      \norm{b - \tilde b}_{\L1 ([0,t] \times \Omega; \R)}
    \right),
  \end{displaymath}
  where $\mathcal{C} \in \L\infty ([0,T];\R_+)$ depends on $\Omega$,
  $K_\alpha$, $K_\beta$, $K_v$, on the map $C_v$, on norms and total
  variation of the functions $a, \tilde a$ and $b, \tilde b$ and of
  the initial data.

  \paragraph{Positivity:} If for all $(t,x) \in [0,T] \times \Omega$,
  $a (t,x) \geq 0$ and $b (t,x) \geq 0$, then for all initial datum
  $(u_o, w_o)$ with $u_o (x) \geq 0$ and $w_o (x) \geq 0$ for all
  $x \in \Omega$, the solution $(u, w)$ is such that $u (t,x) \geq 0$
  and $w (t,x) \geq 0$ for all $(t,x) \in [0,T] \times \Omega$.
\end{theorem}

\noindent The proof is deferred to Section~\ref{sec:proofs}.

The lower semicontinuity of the total variation with respect to the
$\L1$ distance ensures moreover that bounds on the total variation of
the solution can be obtained by means of~\eqref{eq:w_TV}
and~\eqref{eq:u_TV}.

\section{Proofs}
\label{sec:proofs}

In the proofs below, we provide all details wherever necessary and
precise references for those part that differ only slightly from the
cases under consideration.

\subsection{Parabolic Estimates}
\label{sec:parabolicDirichlet}

Fix $T,\mu >0$ and let $\Omega$ satisfy~\ref{it:omega}. This paragraph
is devoted to the IBVP
\begin{equation}
  \label{eq:2}
  \left\{
    \begin{array}{l@{\qquad}r@{\,}c@{\,}l}
      \partial_t w = \mu \, \Delta w + B (t,x) \, w + b (t,x)
      & (t,x)
      & \in
      & [0,T] \times \Omega
      \\
      w (t,\xi) = 0
      & (t,\xi)
      & \in
      & [0,T] \times \partial\Omega
      \\ w (0,x) =
      w_o (x)
      & x
      & \in
      & \Omega\,.
    \end{array}
  \right.
\end{equation}

The following definition is adapted from~\cite{QuittnerSouplet}, see
Remark~\ref{rem:l1delta}.

\begin{definition}
  \label{def:Para}
  A map $w \in \C0 ([0,T]; \L1 (\Omega{;\reali}))$ is a solution
  to~\eqref{eq:2} if $w (0) = w_o$ and for all test functions
  $\phi \in \C2 ([0,T]\times\overline{\Omega};\reali)$ such that
  $\phi (T,x) =0$ for all $x \in \Omega$ and $\phi (t,\xi) = 0$ for
  all $(t, \xi) \in [0,T] \times \partial\Omega$:
  \begin{equation}
    \label{eq:4}
    \begin{array}{@{}r@{}}
      \displaystyle
      \int_0^T \int_\Omega \left(
      w (t,x) \, \partial_t \phi (t,x)
      + \mu \, w (t,x) \, \Delta\phi (t,x)
      + \left(B (t,x) \, w (t,x) + b (t,x)\right) \phi (t,x)
      \right)
      \d{x} \d{t }
      \\
      \displaystyle
      + \int_\Omega w_o (x) \, \phi (0,x) \d{x}
      =
      0 \,.
    \end{array}
  \end{equation}
\end{definition}

\begin{remark}
  \label{rem:l1delta}
  Let
  $d (x, \partial\Omega) = \inf_{y \in \partial\Omega} \norma{x-y}$.
  Recall
  $\norma{w}_{\L1_\delta (\Omega;\reali)} = \int_\Omega \modulo{u (x)}
  \, d (x,\partial\Omega) \d{x}$
  from~\cite[Appendix~B]{QuittnerSouplet}. Since
  $\norma{w}_{\L1_\delta (\Omega;\reali)} \leq \O \norma{w}_{\L1
    (\Omega;\reali)}$, a solution in the sense of
  Definition~\ref{def:Para} is also a \emph{weak $\L1_\delta$
    solution} in the sense of~\cite[Definition~48.8,
  Appendix~B]{QuittnerSouplet}.
\end{remark}

\begin{remark}
  \label{rem:c2}
  In~\Cref{def:Para} it is sufficient to consider test functions
  $\phi \in \C1 ([0,T]\times\overline{\Omega}; \reali)$ such that for
  all $t \in [0,T]$, the map $x \mapsto \phi (t,x)$ is of class
  $\C2 (\overline{\Omega}; \reali)$ and moreover $\phi (T,x) =0$ for
  all $x \in \Omega$ and $\phi (t,\xi) = 0$ for all
  $(t,\xi) \in [0,T]\times \partial\Omega$. This is proved through a
  standard regularization by means of a convolution with a mollifier
  supported in $\reali_-$.
\end{remark}

For $\mu>0$, the heat kernel is denoted
$H_\mu (t,x) = (4 \, \pi \, \mu \, t)^{-n/2} \; \exp
\left(-\norma{x}^2 \middle/(4 \, \mu \, t)\right)$, where $t > 0$,
$x \in \reali^n$. As it is well known,
$\norma{H_\mu (t)}_{\L1 (\reali^n; \reali)} = 1$.

\begin{proposition}
  \label{prop:stimePara}
  Let $\Omega$ satisfy~\ref{it:omega} and fix $\mu>0$. Then, there
  exists a Green function
  \begin{displaymath}
    G \in
    \C\infty(\mathopen]0, +\infty\mathclose[ \times \mathopen]0, +\infty\mathclose[ \times
    \Omega \times \Omega; \reali_+)
    \cap
    \C0(\mathopen]0, +\infty\mathclose[ \times \mathopen]0, +\infty\mathclose[ \times
    \overline{\Omega} \times \overline{\Omega}; \reali_+)
  \end{displaymath}
  such that:
  \begin{enumerate}[label={\bf(G\arabic*)}]

  \item\label{item:K2} For all $t,\tau \in \reali_+$ and
    $x,y \in \Omega$, $G (t,\tau,x,y) = G (t,\tau,y,x)$.

  \item\label{item:K3} For all $t \in \reali_+$,
    $\xi \in \partial\Omega$ and $y \in \Omega$, $G (t,\tau,\xi,y)=0$.

  \item\label{item:K4} There exist positive constants $C,c$ such that
    for all $t \in \reali_+$ and for all $x,y \in \Omega$,
    \begin{align*}
      0  \le
      G (t,\tau,x,y)
      \le \
      &  H_\mu (t-\tau,x-y)
      \\
      \modulo{\partial_t G (t,\tau,x,y)}
      \le \
      & c \, (t-\tau)^{-(n+2)/2}
        \exp\left(-C \, \norma{x-y}^2 \middle/ (t-\tau)\right)
      \\
      \norma{\nabla_x G (t,\tau,x,y)}
      \le \
      & c \, (t-\tau)^{-(n+1)/2}
        \exp\left(-C \, \norma{x-y}^2 \middle/ (t-\tau)\right) \,.
    \end{align*}
  \item\label{item:K5} For all $b \in \L1 ([0,T]\times\Omega; \reali)$
    and all $w_o \in \L1 (\Omega;\R)$, the IBVP
    \begin{equation}
      \label{eq:3}
      \left\{
        \begin{array}{l@{\qquad}r@{\,}c@{\,}l}
          \partial_t w = \mu \, \Delta w + b (t,x)
          & (t,x)
          & \in
          & [0,T] \times  \Omega
          \\
          w (t,\xi) = 0
          & (t,\xi)
          & \in
          & [0,T] \times \partial\Omega
          \\ w (0,x) =
          w_o (x)
          & x
          & \in
          & \Omega
        \end{array}
      \right.
    \end{equation}
    admits a unique solution in the sense of~\Cref{def:Para}, which is
    \begin{equation}
      \label{eq:7}
      w (t,x)
      =
      \int_\Omega G (t,0,x,y) \, w_o (y) \d{y}
      +
      \int_0^t \int_\Omega G (t,\tau,x,y) \, b (\tau,y) \d{y} \d{\tau}\,.
    \end{equation}
  \end{enumerate}
\end{proposition}

\noindent The Green function depends both on $\mu$ and on $\Omega$
but, for simplicity, we omit this dependence.

\begin{proofof}{\Cref{prop:stimePara}}
  Condition~\ref{item:K2} follows from~\cite[Appendix~B,
  \S~48.2]{QuittnerSouplet}.  Property~\ref{item:K3} comes
  from~\cite[Chapter~IV, \S~16, (16.7)--(16.8) p.~408]{MR0241822}.

  The first bound in~\ref{item:K4} follows from~\cite[Formula~(48.4),
  p.440]{QuittnerSouplet},
  the second and the third one from~\cite[Chapter~IV, \S~16,
  Theorem~16.3, p.~413]{MR0241822}.

  To prove~\ref{item:K5}, use Remark~\ref{rem:l1delta}
  and~\cite[Proposition~48.9, Appendix~B]{QuittnerSouplet},
  \cite[Corollary~48.10, Appendix~B]{QuittnerSouplet} and the Maximum
  Principle~\cite[Proposition~52.7, Appendix~B]{QuittnerSouplet},
  which ensure the equivalence between~\eqref{eq:4} and~\eqref{eq:7}
  as soon as either $w_o \geq 0$, $b \geq 0$ or $w_o \leq 0$,
  $b \leq 0$. The linearity of~\eqref{eq:3} allows to complete the
  proof.
\end{proofof}

\begin{proposition}
  \label{prop:superStimePara}
  Let $\Omega$ satisfy~\ref{it:omega}, fix $\mu>0$ and let
  \begin{enumerate}[label={\bf(P\arabic*)}]
  \item\label{it:P1} $w_o \in \L\infty (\Omega;\reali)$,
  \item\label{it:P2} $B \in \L\infty ([0,T] \times \Omega;\reali)$,
  \item\label{it:P3} $b \in \L1 ([0,T]; \L\infty( \Omega;\reali))$.
  \end{enumerate}
  Then:
  \begin{enumerate}[label={\bf(\arabic*)}]

  \item\label{it:P_existence} Problem~\eqref{eq:2} admits a unique
    solution in the sense of~\Cref{def:Para}.

  \item\label{it:P_formula} The solution to~\eqref{eq:2} is implicitly
    given by
    \begin{equation}
      \label{eq:5}
      \begin{array}{rcl}
        \displaystyle
        w (t,x)
        & =
        & \displaystyle
          \int_\Omega G (t,0,x,y) \, w_o (y) \d{y}
        \\
        &
        & \displaystyle
          \qquad +
          \int_0^t \int_\Omega G (t,\tau,x,y)
          \left(B (\tau,y) \, w (\tau,y) + b (\tau,y)\right)
          \d{y} \d\tau
      \end{array}
    \end{equation}
    where $G$, independent of $b$ and $B$, is defined
    in~\Cref{prop:stimePara}.

  \item\label{it:P_apriori} The following \emph{a priori} bounds hold
    for all $t \in [0,T]$
    \begin{align}
      \label{eq:11}
      \norma{w (t)}_{\L1 (\Omega;\reali)}
      \leq
      & \left(
        \norma{w_o}_{\L1 (\Omega;\reali)}
        +
        \norma{b}_{\L1 ([0,t]\times \Omega;\reali)}
        \right)
        \exp \int_0^t \norma{B(\tau)}_{\L\infty (\Omega;\reali)}
        \d\tau,
      \\
      \label{eq:32}
      \norma{w (t)}_{\L\infty (\Omega;\reali)}
      \leq
      & \left(
        \norma{w_o}_{\L\infty (\Omega;\reali)}
        +
        \norma{b}_{\L1 ([0,t]; \L\infty (\Omega;\reali))}
        \right)
        \exp \int_0^t \norma{B(\tau)}_{\L\infty (\Omega;\reali)}
        \d\tau \,.
    \end{align}

  \item\label{it:P_stability} If $w_1$, $w_2$ solve~\eqref{eq:2} with
    data $w_o^1$, $w_o^2$ satisfying~\ref{it:P1}, functions $B_1$,
    $B_2$ satisfying~\ref{it:P2} and functions $b_1$, $b_2$
    satisfying~\ref{it:P3}, then
    \begin{equation}
      \label{eq:10}
      \begin{array}{rcl}
        &
        & \displaystyle
          \norma{w_1 (t) -w_2 (t)}_{\L1 (\Omega;\reali)}
        \\
        & \leq
        & \displaystyle
          \left(
          \norma{w_o^1- w_o^2}_{\L1 (\Omega;\reali)}
          +
          \norma{b_1 - b_2}_{\L1 ([0,t]\times\Omega;\reali)}
          \right)
          \exp \int_0^t \norma{B_1(\tau)}_{\L\infty (\Omega;\reali)}
          \d\tau
        \\
        &
        & \displaystyle
          +
          \norma{B_1 - B_2}_{\L1 ([0,t]\times\Omega;\reali)}
          \left(
          \norma{w_o^2}_{\L\infty (\Omega;\reali)}
          +
          \norma{b_2}_{\L1 ([0,t]; \L\infty (\Omega;\reali))}
          \right)
        \\
        &
        & \displaystyle
          \qquad
          \times \exp \int_0^t
          \left(
          \norma{B_1(\tau)}_{\L\infty (\Omega;\reali)}
          +
          \norma{B_2(\tau)}_{\L\infty (\Omega;\reali)}\right)
          \d\tau \,.
      \end{array}
    \end{equation}

  \item\label{it:P_positivity} Positivity: if $b \ge 0$ and
    $w_o \ge 0$ then $w \ge0$.

  \item \label{it:P_tv} If $w_o\in \BV (\Omega; \R)$ and
    $b(t) \in \BV (\Omega; \R)$ for all $t \in [0,T]$, then for all
    $t \in [0,T]$ the following estimate holds.
    \begin{align}
      \nonumber
      & \tv \left(w(t)\right)
      \\       \label{eq:37}
      \le \ & \tv (w_o) + \int_0^t \tv\left(b
              (\tau)\right) \dd\tau
      \\ \nonumber
      & + \mathcal{O} (1) \, \sqrt{t}\, \norma{B}_{\L\infty
        ([0,t]\times \Omega; \R)} \left( \norma{w_o}_{\L1 (\Omega; \R)} {+}
        \norma{b}_{\L1 ([0,t]\times\Omega; \R)} \right)
        \exp \int_0^t \norma{B(\tau)}_{\L\infty
        (\Omega;R)} \dd\tau.
    \end{align}
  \end{enumerate}
\end{proposition}

\noindent We note, for completeness, that in the setting of
Proposition~\ref{prop:superStimePara} the following regularity results
-- not of use in the sequel -- can also be obtained:
\begin{description}
\item[(7)] If $w_o\in \Cc1 (\Omega; \R)$ then the solution $w$ is such
  that $w(t) \in \C1 (\Omega; \R)$ for all $t \in [0,T]$.

\item[(8)] The solution $w$ is H\"older continuous in time.
\end{description}
\begin{proofof}{\Cref{prop:superStimePara}}
  We split the proof in a few steps.

  \paragraph{Claim~1: Problem~\eqref{eq:2} admits at most one solution
    in the sense of~\Cref{def:Para}.}

  Observe that if $w_1$, $w_2$ solve~\eqref{eq:2} in the sense
  of~\Cref{def:Para}, then their difference satisfies
  $\int_0^T \int_\Omega (w_2-w_1) \, (\partial_t \phi + \mu \, \Delta
  \phi + B \, \phi) \d{x} \d{t} = 0$, for all $\phi$ as regular as
  specified in~\Cref{rem:c2}. By~\cite[(ii) in~Theorem~48.2,
  Appendix~B]{QuittnerSouplet}, we choose as $\phi$ the strong
  solution to
  \begin{displaymath}
    \left\{
      \begin{array}{l@{\qquad}r@{\,}c@{\,}l}
        \partial_t \phi + \mu \, \Delta\phi + B (t,x) \, \phi = f
        & (t,x)
        & \in
        & [0,T] \times \Omega
        \\
        \phi (t,\xi) = 0
        & (t,\xi)
        & \in
        & [0,T] \times \partial\Omega
        \\
        \phi (T,x) = 0
        & x
        & \in
        & \Omega
      \end{array}
    \right.
  \end{displaymath}
  where $f \in \C0 ([0,T] \times \overline{\Omega};\reali)$. We thus
  have $\int_0^T \int_\Omega (w_2-w_1) \, f \d{x} \d{t} = 0$, so that,
  by the arbitrariness of $f$, $w_1=w_2$.

  \paragraph{Claim~2: If $w \in \L\infty ([0,T]; \L1 (\Omega;\reali))$
  satisfies~\eqref{eq:5}, then~\eqref{eq:11} and~\eqref{eq:32} hold.}
Consider first~\eqref{eq:11}. By~\ref{item:K4}, recalling
$\norma{H_\mu (t)}_{\L1 (\Omega; \reali)} \leq 1$, we have
\begin{align*}
  \norma{w (t)}_{\L1 (\Omega; \reali)}
  \le \
  & \int_\Omega \int_\Omega
    G(t,0,x,y) \modulo{w_o(y)}\d{y} \d{x}
  \\
  & + \int_\Omega \int_0^t \int_\Omega
    G(t,\tau,x,y) \modulo{B (\tau,y) \, w (\tau,y) + b (\tau,y)}
    \d{y} \dd\tau \d{x}
  \\
  \le \
  & \int_\Omega \int_\Omega
    H_\mu(t,x-y) \modulo{w_o(y)} \d{y} \d{x}
  \\
  & + \int_\Omega \int_0^t \int_\Omega
    H_\mu(t-\tau,x-y) \modulo{B (\tau,y) \, w (\tau,y) + b (\tau,y)}
    \d{y} \d\tau \d{x}
  \\
  \le \
  & \norma{w_o}_{\L1 (\Omega;\reali)}
    +
    \int_0^t \norma{B(\tau)}_{\L\infty (\Omega;\reali)} \,
    \norma{w (\tau)}_{\L1 (\Omega;\reali)}  \dd\tau
    + \norma{b}_{\L1 ([0,t] \times \Omega;\reali)} \,.
\end{align*}
An application of Gronwall Lemma~\cite[Lemma~3.1]{bressan-piccoli}
yields~\eqref{eq:11}. The proof of~\eqref{eq:32} is entirely similar.

\paragraph{Claim~3: If
  $w_1, w_2 \in \L\infty ([0,T]; \L1 (\Omega;\reali))$
  satisfy~\eqref{eq:5}, then~\eqref{eq:10} holds.}

Note that
\begin{align*}
  w_1 (t,x) - w_2 (t,x)
  =\
  &   \int_\Omega G (t,0,x,y) \, \left(w_o^1 (y)- w_o^2 (x) \right)\d{y}
  \\
  &  \qquad +
    \int_0^t \int_\Omega G (t,\tau,x,y)
    \left(B_1 (\tau,y) \, w_1 (\tau,y) - B_2 (\tau,y) \, w_2 (\tau,y) \right)
    \d{y} \d\tau
  \\
  & \qquad +
    \int_0^t \int_\Omega G (t,\tau,x,y)
    \left(b_1 (t,y) - b_2 (t,y) \right)
    \d{y} \d\tau
  \\
  = \
  & \int_\Omega G (t,0,x,y) \, \left(w_o^1 (y)- w_o^2 (x) \right)\d{y}
  \\
  &   \qquad +
    \int_0^t \int_\Omega G (t,\tau,x,y)
    B_1 (\tau,y) \, \left(w_1 (\tau,y) - w_2 (\tau,y) \right)
    \d{y} \d\tau
  \\
  &   \qquad +
    \int_0^t \int_\Omega G (t,\tau,x,y) \, \tilde b (t,y)
    \d{y} \d\tau,
\end{align*}
where
$\tilde b (t,x) = \left(B_1 (t,x) - B_2 (t,x)\right) \, w_2 (t,x) +
b_1 (t,x) - b_2 (t,x)$.  Proceeding as in the proof of {\bf Claim~2}
and exploiting~\eqref{eq:32}, we obtain
\begin{align*}
  & \norma{w_1 (t) - w_2 (t)}_{\L1 (\Omega; \reali)}
  \\
  \leq \
  & \left(
    \norma{w_o^1- w_o^2}_{\L1 (\Omega;\reali)}
    +
    \norma{\tilde b}_{\L1 ([0,t] \times \Omega;\reali)}
    \right)
    \exp \int_0^t \norma{B_1(\tau)}_{\L\infty (\Omega;\reali)}
    \d\tau
  \\
  \leq \
  & \left(
    \norma{w_o^1- w_o^2}_{\L1 (\Omega;\reali)}
    +
    \norma{B_1 - B_2}_{\L1 ([0,t]\times\Omega;\reali)}
    \norma{w_2}_{\L\infty ([0,t]\times\Omega;\reali)}
    +
    \norma{b_1 - b_2}_{\L1 ([0,t]\times\Omega;\reali)}
    \right)
  \\
  & \quad \times
    \exp \int_0^t \norma{B_1(\tau)}_{\L\infty (\Omega;\reali)}
    \d\tau
  \\
  \leq \
  & \left(
    \norma{w_o^1- w_o^2}_{\L1 (\Omega;\reali)}
    +
    \norma{b_1 - b_2}_{\L1 ([0,t]\times\Omega;\reali)}
    \right)
    \exp \int_0^t \norma{B_1(\tau)}_{\L\infty (\Omega;\reali)}
    \d\tau
  \\
  & \quad +
    \norma{B_1 - B_2}_{\L1 ([0,t]\times\Omega;\reali)}
    \left(
    \norma{w_o^2}_{\L\infty (\Omega;\reali)}
    +
    \norma{b_2}_{\L1 ([0,t]; \L\infty (\Omega;\reali))}
    \right)
  \\
  & \qquad
    \times \exp \int_0^t
    \left(
    \norma{B_1(\tau)}_{\L\infty (\Omega;\reali)}
    +
    \norma{B_2(\tau)}_{\L\infty (\Omega;\reali)}\right)
    \d\tau \,.
\end{align*}

\paragraph{Claim~4: If $w \in \L\infty ([0,T]; \L1 (\Omega;\reali))$
  satisfies~\eqref{eq:5}, then
  $w \in \C0([0,T]; \L1 (\Omega;\reali))$.}

Introduce the abbreviation
$\tilde{b} (t,x) = B (t,x) \, w (t,x) + b (t,x)$ so that,
using~\eqref{eq:11},
\begin{align*}
  \norma{\tilde{b} (t)}_{\L1(\Omega; \reali)}
  \leq \
  & \norma{B (t)}_{\L\infty (\Omega;\reali)} \,
    \norma{w (t)}_{\L1 (\Omega;\reali)}
    +
    \norma{b}_{\L1 (\Omega;\reali)}
  \\
  \leq \
  & \norma{B (t)}_{\L\infty (\Omega;\reali)} \,
    \left(
    \norma{w_o}_{\L1 (\Omega;\reali)}
    +
    \norma{b}_{\L1 ([0,t] \times \Omega;\reali)}
    \right)
    \exp  \int_0^t \norma{B(\tau)}_{\L\infty (\Omega;\reali)} \d\tau
  \\
  &
    +
    \norma{b (t)}_{\L1 (\Omega;\reali)}
\end{align*}
and
$\norma{\tilde{b}}_{\L\infty ([0,t]; \L1 (\Omega;\reali))} \leq \O$.
Compute, using~\ref{item:K4}, for $t_2 > t_1 >0$ and
$s,\sigma \in\mathopen]t_1, t_2\mathclose[$,
\begin{align*}
  & \norma{w (t_2)-w (t_1)}_{\L1 (\Omega;\reali)}
  \\
  \le \
  & \int_\Omega \int_\Omega \modulo{G (t_2,0,x,y) - G (t_1,0,x,y)} \modulo{w_o (y)} \d{y} \d{x}
  \\
  & + \int_0^{t_1} \int_{\Omega} \int_{\Omega}
    \modulo{G (t_2,\tau,x,y) - G (t_1,\tau,x,y)} \modulo{\tilde{b} (\tau,y)}
    \d{y} \d{x} \d{\tau}
  \\
  & +
    \int_{t_1}^{t_2}  \int_{\Omega} \int_{\Omega}
    \modulo{G (t_2,\tau,x,y)} \modulo{\tilde{b} (\tau,y)}
    \d{y} \d{x} \d{\tau}
  \\
  \le \
  & \dfrac{c\, \modulo{t_2-t_1}}{s^{1+n/2}}
    \int_\Omega \int_\Omega \modulo{w_o (y)}
    \exp\left(-C \, \norma{x-y}^2 \middle/ s\right) \d{y} \d{x}
  \\
  & +
    \dfrac{c\, \modulo{t_2-t_1}}{\sigma^{1+n/2}}
    \int_{0}^{t_1} \int_\Omega \int_\Omega
    \modulo{\tilde{b} (t,x)} \,
    \exp\left(-C \, \norma{x-y}^2 \middle/ \sigma\right)
    \d{y} \d{x} \d{\tau}
  \\
  & +
    c \int_{t_1}^{t_2} (t_2-\tau)^{-n/2}
    \int_\Omega \int_\Omega
    \modulo{\tilde{b} (t,x)} \,
    \exp\left(-C \, \norma{x-y}^2
    \middle/ (t_2-\tau)\right)
    \d{y} \d{x} \d{\tau}
  \\
  \le \
  & \dfrac{c\, \modulo{t_2-t_1}}{s^{1+n/2}} \,
    \norma{w_o}_{\L1 (\Omega;\reali)}
    \int_\Omega \exp\left(-C \, \norma{x}^2 \middle/ s\right) \d{x}
  \\
  & + \dfrac{c\, \modulo{t_2-t_1}}{\sigma^{1+n/2}} \, t_1 \,
    \norma{\tilde{b}}_{\L\infty ([0,t_2]; \L1(\Omega; \reali))}
    \int_\Omega \exp\left(-C \, \norma{x}^2 \middle/ \sigma\right) \d{x}
  \\
  & +
    c \, \norma{\tilde{b}}_{\L\infty ([0,t_2]; \L1(\Omega; \reali))}
    \int_{t_1}^{t_2} (t_2-\tau)^{-n/2}
    \int_\Omega
    \exp\left(-C \, \norma{x}^2
    \middle/ (t_2-\tau)\right)
    \d{x}   \d{\tau}
  \\
  \le \
  & \dfrac{c\, \modulo{t_2-t_1}}{C^{n/2} \, s} \,
    \norma{w_o}_{\L1 (\Omega;\reali)} \,
    \int_{\reali^n} \exp\left(- \norma{x}^2 \right) \d{x}
  \\
  &+ \dfrac{c\, \modulo{t_2-t_1}}{C^{n/2}} \,
    \norma{\tilde{b}}_{\L\infty([0,t_2];\L1(\Omega; \reali))} \,
    \int_{\reali^n} \exp\left(- \norma{x}^2 \right) \d{x}
  \\
  & +
    \dfrac{c\, \modulo{t_2-t_1}}{C^{n/2}} \,
    \norma{\tilde{b}}_{\L\infty ([0,t_2]; \L1(\Omega; \reali))} \,
    \int_{\reali^n} \exp\left(- \norma{x}^2 \right) \d{x}
  \\
  \le \
  & \O  \left(1+\dfrac{1}{s}\right) \, \modulo{t_2-t_1} \,.
\end{align*}
Since $(t_2-t_1)/s \leq t_2/t_1 - 1 < \delta$ as soon as
$t_2 < t_1 +\delta$, the $\L1 (\Omega;\reali)$ continuity of $w$ is
proved.

\paragraph{Claim~5: There exists a solution to~\eqref{eq:2} in the
  sense of~\Cref{def:Para} satisfying~\eqref{eq:5}.} Assume first that
$w_o \in \C0 (\overline{\Omega}; \reali)$ with $w_o = $ on
$\partial\Omega$, $B \in \C0 ([0,T]\times\overline{\Omega}; \reali)$
and $b \in \C0 ([0,T]\times \overline{\Omega}; \reali)$.
From~\cite[Chapter~IV, \S~16]{MR0241822} we know that~\eqref{eq:2}
admits a classical solution, say $w$. Define
$\tilde{b} (t,x) = B (t,x) \, w (t,x) + b (t,x)$, so that $w$
satisfies~\eqref{eq:7} by~\ref{item:K5}
in~\Cref{prop:stimePara}. Hence, $w$ also satisfies~\eqref{eq:5}.

Under the weaker regularity~\ref{it:P1}, \ref{it:P2} and~\ref{it:P3},
introduce sequences $w_o^\nu \in \C0 (\overline{\Omega}; \reali)$ and
$B^\nu,b^\nu \in \C0 ([0,T]\times \overline{\Omega}; \reali)$
converging to $w_o$ in $\L1 (\Omega;\reali)$ and to $B,b$ in
$\L1 ([0,T] \times \Omega; \reali)$. Call $w^\nu$ the corresponding
classical solution to~\eqref{eq:2} which, by the paragraph above,
exists and satisfies
\begin{equation}
  \label{eq:8}
  w^\nu (t,x) =\!\! \int_\Omega\! G (t,0,x,y) w^\nu_o (y) \d{y}
  +
  \!\!\int_0^t\!\! \int_\Omega \!G (t,\tau,x,y)\!
  \left(B^\nu \!(\tau,y)  w^\nu (t,y)
    + b^\nu (t,y)\right)\!\d{y} \d{\tau}.
\end{equation}
Hence, by {\bf Claim 3}, $w^\nu$ and $w^{\nu+1}$
satisfy~\eqref{eq:10}, so that
\begin{align*}
  & \norma{w^{\nu+1} (t) -w^\nu (t)}_{\L1 (\Omega;\reali)}
  \\
  \le \
  & \left( \norma{w_o^{\nu+1} - w_o^\nu}_{\L1 (\Omega;\reali)}
    +
    \norma{b^{\nu+1} - b^\nu }_{\L1 ([0,t] \times \Omega;\reali)}
    \right)
  \\
  & \times
    \exp\left(\int_0^t
    \norma{B^{\nu} (\tau)}_{\L\infty (\Omega;\reali)}
    \d\tau
    \right)
  \\
  & +
    \norma{B^{\nu+1} - B^\nu}_{\L1 ([0,t]\times\Omega;\reali)}
    \left(
    \norma{w_o^{\nu+1}}_{\L\infty (\Omega;\reali)}
    +
    \norma{b^{\nu+1}}_{\L1 ([0,t]; \L\infty (\Omega;\reali))}
    \right)
  \\
  & \qquad
    \times \exp \int_0^t
    \left(
    \norma{B^\nu(\tau)}_{\L\infty (\Omega;\reali)}
    +
    \norma{B^{\nu+1}(\tau)}_{\L\infty (\Omega;\reali)}\right)
    \d\tau \,.
\end{align*}
By the hypotheses on the sequences $w_o^\nu$, $B^\nu$ and $b^\nu$,
$w^\nu $ is a Cauchy sequence in $\L1 ([0,T] \times \Omega; \reali)$
converging to a function $w$ in $\L1 ([0,T] \times \Omega; \reali)$.
Moreover, since, by~\eqref{eq:11},
\begin{displaymath}
  \norma{w^\nu (t)}_{\L1 (\Omega;\reali)}
  \leq
  \left(
    \norma{w_o^\nu}_{\L1 (\Omega;\reali)}
    +
    \norma{b^\nu}_{\L1 ([0,t] \times \Omega;\reali)}
  \right)
  \exp \int_0^t \norma{B^\nu(\tau)}_{\L\infty (\Omega;\reali)}
  \d\tau
\end{displaymath}
letting $\nu \to +\infty$, we also have
\begin{equation}
  \label{eq:56}
  \norma{w (t)}_{\L1 (\Omega;\reali)}
  \leq
  \left(
    \norma{w_o}_{\L1 (\Omega;\reali)}
    +
    \norma{b}_{\L1 ([0,t] \times \Omega;\reali)}
  \right)
  \exp \int_0^t \norma{B(\tau)}_{\L\infty (\Omega;\reali)}
  \d\tau
\end{equation}
and hence $w \in \L\infty ([0,T]; \L1 (\Omega;\reali))$.

Passing to the limit in~\eqref{eq:8}, by the Dominated Convergence
Theorem, we get that $w$ satisfies~\eqref{eq:5} for
a.e.~$x \in \Omega$.  Moreover, for any
$\phi \in \C2 ([0,T] \times \overline{\Omega}; \reali)$, a further
application of the Dominated Convergence Theorem allows to pass to the
limit $\nu \to +\infty$ in~\eqref{eq:4}, proving that $w$ satisfies
also~\eqref{eq:4}. The $\C0$ in time -- $\L1$ in space continuity
required by~\Cref{def:Para} is proved in \textbf{Claim~4}.

This completes the proof of~\ref{it:P_existence} and
proves~\ref{it:P_formula}. Then, \ref{it:P_apriori} follows
from~\eqref{eq:56} and~\ref{it:P_stability} is proved similarly, as in
\textbf{Claim~2}.

\paragraph{Claim~6: Positivity.}
As above, consider a more regular and non negative datum
$w_o \in \C0 (\overline{\Omega}; \reali_+)$ with $w_o = 0$ on
$\partial\Omega$, $B \in \C0 ([0,T]\times\overline{\Omega}; \reali)$
and a non negative
$b \in \C0 ([0,T]\times \overline{\Omega}; \reali_+)$.
From~\cite[Chapter~IV, \S~16]{MR0241822} we know that~\eqref{eq:2}
admits a classical solution, say $w$. By~\cite[Chapter~I, \S~2,
Theorem~2.1]{MR0241822}, we also know that $w \geq 0$. Continue as in
the proof of Claim~5 to obtain that in the general case the solution is point-wise almost everywhere limit of non negative classical solutions, completing the proof of~\ref{it:P_positivity}..

\paragraph{Claim~7: $\BV$-bound.} We follow the idea
of~\cite[Proposition~2]{SAPM2021}. First, regularize the initial datum
$w_o$ and the function $b$ appearing in the source term as follows:
there exist sequences $w_o^h\in \C\infty (\Omega; \R)$ and
$b_h (t)\in \C\infty (\Omega;\R)$, for all $t \in [0,T]$, such that
\begin{align*}
  \lim_{h \to +\infty} \norma{w_o^h - w_o}_{\L1 (\Omega; \R)} = \
  & 0,
  &
    \norma{w_o^h}_{\L\infty (\Omega; \R)} \le \
  & \norma{w_o}_{\L\infty (\Omega; \R)},
  &
    \tv(w_o^h) \le\
  & \tv(w_o),
\end{align*}
and for all $t \in [0,T]$
\begin{align*}
  \lim_{h \to +\infty} \norma{b_h (t) - b (t)}_{\L1 (\Omega; \R)} = \
  & 0,
  &
    \norma{b_h (t)}_{\L\infty (\Omega; \R)} \le \
  & \norma{b (t)}_{\L\infty (\Omega; \R)},
  &
    \tv(b_h (t)) \le\
  & \tv(b(t)).
\end{align*}
According to~\eqref{eq:5}, define the sequence $w_h$ corresponding to
the sequences $w_o^h$ and $b_h$. By construction and due to the
regularity of the Green function $G$,
$w_h (t) \in \C\infty (\Omega; \R)$ for all $t \in [0,T]$. Moreover,
exploiting~\eqref{eq:10}, it follows immediately that
$w_h (t) \to w(t)$ in $\L1 (\Omega; \R)$ as $h \to + \infty$ for
a.e.~$t \in [0,T]$. Compute $\nabla w_h$, using~\eqref{eq:5}, the
symmetry property of the Green function $G$, see~\ref{item:K2}
in~\Cref{prop:stimePara}, integration by parts and~\ref{item:K3}
in~\Cref{prop:stimePara}:
\begin{align*}
  \nabla w_h(t,x) = \
  & \int_\Omega \nabla_x G (t,0,x,y) \, w_o^h(y) \dd{y}
    + \int_0^t\int_\Omega  \nabla_x G (t,\tau,x,y) \, B (\tau,y) \, w_h (\tau,y)
    \dd{y}\dd\tau
  \\
  & + \int_0^t\int_\Omega  \nabla_x G (t,\tau,x,y) \, b_h(\tau,y) \dd{y}\dd\tau
  \\
  = \
  & \int_\Omega \nabla_y G (t,0,y,x) \, w_o^h(y) \dd{y}
    + \int_0^t\int_\Omega  \nabla_x G (t,\tau,x,y) \, B (\tau,y) \, w_h (\tau,y)
    \dd{y}\dd\tau
  \\
  & + \int_0^t\int_\Omega  \nabla_y G (t,\tau,y,x) \, b_h(\tau,y) \dd{y}\dd\tau
  \\
  = \
  & - \int_\Omega  G (t,0,y,x) \, \nabla w_o^h(y) \dd{y}
    + \int_0^t\int_\Omega  \nabla_x G (t,\tau,x,y) \, B (\tau,y) \, w_h (\tau,y)
    \dd{y}\dd\tau
  \\
  & - \int_0^t\int_\Omega  G (t,\tau,y,x) \, \nabla b_h(\tau,y) \dd{y}\dd\tau.
\end{align*}
Pass now to the $\L1$-norm, exploiting~\ref{item:K4}
in~\Cref{prop:stimePara} and~\eqref{eq:11}:
\begin{align*}
  & \norma{\nabla w_h(t)}_{\L1 (\Omega; \R)}
  \\
  \le \
  & \norma{\nabla w_o^h}_{\L1 (\Omega; \R)}
    + \int_0^t \norm{\nabla b_h (\tau)}_{\L1 (\Omega; \R)}\dd\tau
  \\
  & + \int_0^t \frac{c}{(t-\tau)^{-(n+1)/2}}
    \norma{w_h(\tau)}_{\L1 (\Omega; \R)} \norma{B(\tau)}_{\L\infty (\Omega; \R)}
    \int_\Omega \exp\left(-C\norma{x-y}^2\middle/ (t-\tau)\right)\dd{y}\dd\tau
  \\
  \le \
  &  \norma{\nabla w_o^h}_{\L1 (\Omega; \R)}
    + \int_0^t \norm{\nabla b_h (\tau)}_{\L1 (\Omega; \R)}\dd\tau
  \\
  & + \mathcal{O} (1) \sqrt{t}
    \left(\norma{w_o^h}_{\L1 (\Omega; \R)} + \norm{b_h}_{\L1 ([0,t]\times\Omega; \R)}
    \right)
    \norma{B}_{\L\infty ([0,t]\times\Omega; \R)}
    \exp\int_0^t \norm{B (\tau)}_{\L\infty (\Omega;\R)}\dd\tau.
\end{align*}
By the lower semicontinuity of the total variation and the hypotheses
on the regularizing sequences $w_o^h$ and $b_h$, passing to the limit
$h \to +\infty$ yields~\eqref{eq:37}, proving~\ref{it:P_tv}.
\end{proofof}

\subsection{Hyperbolic Estimates}
\label{sec:hyperbolic}

Fix $T >0$. This paragraph is devoted to the IBVP
\begin{equation}
  \label{eq:12}
  \left\{
    \begin{array}{l@{\qquad}r@{\,}c@{\,}l}
      \partial_t u + \div\left(c (t,x) \, u\right)
      = A (t,x) \, u + a (t,x)
      & (t,x)
      & \in
      & [0,T] \times \Omega
      \\
      u (t,\xi) = 0
      & (t,\xi)
      & \in
      & \mathopen]0,T] \times \partial \Omega
      \\
      u (0,x) = u_o (x)
      & x
      & \in
      & \Omega \,.
    \end{array}
  \right.
\end{equation}
We assume throughout the following conditions:
\begin{enumerate}[label={\bf(H\arabic*)}]
\item\label{item:H0}
  $u_o \in \left(\L\infty \cap \BV \right) (\Omega;\reali)$
\item\label{it:H1}
  $c \in \left(\C0 \cap \L\infty\right) ([0,T] \times
  \Omega;\reali^n)$, $c(t) \in \C1 (\Omega; \R^n)$ for all
  $t \in [0,T]$,
  $D_x c \in \L\infty ([0,T] \times \Omega; \R^{n\times n})$.
\item\label{it:H2} $A \in \L\infty ([0,T] \times \Omega;\reali)$ and
  for all $t \in [0,T]$, $A (t) \in \BV (\Omega;\reali)$.
\item\label{it:H3}
  $a \in \L1 \left([0,T]; \L\infty(\Omega; \R)\right)$ and for all
  $t \in [0,T]$, $a (t) \in \BV (\Omega;\reali)$.
\end{enumerate}

\noindent Note that~\ref{it:H1} ensures that
$c (t) \in \C{0,1} (\overline{\Omega};\reali^n)$ for any $t \in[0,T]$.

For $(t_o,x_o) \in [0,T] \times \overline{\Omega}$ introduce the map
\begin{equation}
  \label{eq:caratt}
  \begin{array}{lccl}
    X(\,\cdot\,; t_o,x_o) :
    & I(t_o,x_o)
    & \to
    & \overline{\Omega}
    \\
    & t
    & \to
    & X(t; t_o,x_o)
  \end{array}
  \quad \text{ solving }
  \quad
  \begin{cases}
    \dot x = c(t,x), \\
    x(t_o) = x_o,
  \end{cases}
\end{equation}
$I(t_o,x_o)$ being the maximal interval where a solution to the Cauchy
problem in~\eqref{eq:caratt} is defined (with values in
$\overline{\Omega}$). For $t \in [0,T]$ and for $x \in \Omega$ define
\begin{equation}
  \label{eq:13}
  \mathcal{E} (\tau,t,x) = \exp \left(
    \int_\tau^t\left(
      A \left(s, X(s;t,x)\right) - \div c\left(s, X (s;t,x)\right)
    \right)\d{s}
  \right)
\end{equation}
and for all $(t,x) \in \mathopen]0,T] \times \Omega$, if
$x \in X(t; [0,t\mathclose[, \partial\Omega) \cap \Omega$, set
\begin{equation}
  \label{eq:14}
  T(t,x) = \inf\{s \in [0,t] \colon X(s;t,x) \in \Omega\},
\end{equation}
which is well defined by~\ref{it:H1} and Cauchy Theorem. Note that the
well posedness of the Cauchy problem~\eqref{eq:caratt}, ensured
by~\ref{it:H1}, implies that for all $t \in \mathopen]0, T]$
\begin{equation}
  \label{eq:38}
  \Omega
  \subseteq
  X (t;0,\Omega) \cup X (t; [0,t\mathclose[, \partial\Omega)
  \subseteq
  \overline{\Omega}
  \quad \mbox{ and } \quad
  X (t;0,\Omega) \cap X (t; [0,t\mathclose[, \partial\Omega) = \emptyset \,.
\end{equation}

As is well known, integrating~\eqref{eq:12} along characteristics
leads, for $(t,x) \in [0,T] \times \Omega$, to
\begin{equation}
  \label{eq:9}
  u(t,x) =
  \begin{cases}
    \displaystyle u_o\left(X (0;t,x)\right) \mathcal{E} (0,t,x) {+}
    \int_0^ta\left(\tau,X(\tau;t,x)\right)\mathcal{E} (\tau,t,x)
    \d\tau & x \in X(t;0,\Omega)
    \\
    \displaystyle \int_{T(t,x)}^t
    a\left(\tau,X(\tau;t,x)\right)\mathcal{E} (\tau,t,x) \d\tau & x
    \in X(t;[0,t\mathclose[,\partial\Omega)
  \end{cases}
\end{equation}

The following relation will be of use below, see for
instance~\cite[Chapter~3]{bressan-piccoli} for a proof:
\begin{align}
  \label{eq:pxoX}
  D_{x_o} X (t; t_o,x_o) = \
    & M(t),
      \text{ the matrix } M \text{ solves }
      \left\{
      \begin{array}{l}
        \dot M = D_x c\left(t,X (t;t_,x_o)\right) M
        \\
        M(t_o) = \Id.
      \end{array}
  \right.
\end{align}

We first particularize classical estimates to the present case.

\begin{lemma}[{\cite[Lemma~4.2]{teo}}]
  \label{lem:L1}
  Let~\ref{it:omega} and~\ref{it:H1} hold.
  \begin{enumerate}
  \item\label{item:3} Assume $u_o \in \L1 (\Omega;\reali)$,
    $A \in \L\infty ([0,T]\times\Omega; \reali)$ and
    $a \in \L1 ([0,T]\times\Omega; \reali)$. Then, the map $u$ defined
    in~\eqref{eq:9} satisfies for all $t \in [0,T]$
    \begin{displaymath}
      \norma{u (t)}_{\L1 (\Omega;\reali)}
      \leq
      \left(
        \norma{u_o}_{\L1 (\Omega;\reali)}
        +
        \norma{a}_{\L1 ([0,t]\times\Omega; \reali)}
      \right) \,
      \exp\left(\norma{A}_{\L\infty ([0,t]\times\Omega;\reali)} \, t\right) .
    \end{displaymath}

  \item\label{item:4} Assume $u_o \in \L\infty (\Omega;\reali)$,
    $A \in \L1 ([0,T];\L\infty(\Omega;\reali))$ and
    $a \in \L1 \left([0,T]; \L\infty(\Omega; \reali)\right)$. Then,
    the map $u$ defined in~\eqref{eq:9} satisfies for all
    $t \in [0,T]$
    \begin{align*}
      \norma{u (t)}_{\L\infty (\Omega;\reali)}
      \le \
      & \left(
        \norma{u_o}_{\L\infty (\Omega;\reali)}
        +
        \norma{a}_{\L1 ([0,t];\L\infty(\Omega; \reali))}
        \right) \,
      \\
      & \times \exp\left(
        \norma{A}_{\L1 ([0,t];\L\infty(\Omega;\reali))}
        +
        \norma{\div c}_{\L1 ([0,t];\L\infty (\Omega;\reali))}
        \right) .
    \end{align*}
  \end{enumerate}
\end{lemma}

\begin{lemma}
  \label{lem:A}
  Let~\ref{it:omega} and~\ref{it:H1} hold. Assume
  $u_o \in \L\infty (\Omega;\reali)$ and
  $a \in \L1 \left([0,T]; \L\infty(\Omega; \reali)\right)$. Fix
  $A_1,A_2 \in \L\infty ([0,T]\times\Omega;\reali)$. Then, the maps
  $u_1,u_2$ defined in~\eqref{eq:9} satisfy for all $t \in [0,T]$
  \begin{align*}
    \norma{u_2 (t) - u_1 (t)}_{\L1 (\Omega;\reali)}
    \le \
    & \exp \left(
      t
      \max
      \left\{
      \norma{A_1}_{\L\infty ([0,t]\times\Omega;\reali)}
      \,,\;
      \norma{A_2}_{\L\infty ([0,t]\times\Omega;\reali)}
      \right\}\right)
    \\
    & \times
      \left(
      \norma{u_o}_{\L\infty (\Omega;\reali)}
      +
      \norma{a}_{\L1 ([0,t]; \L\infty(\Omega;\reali))}
      \right)
      \norma{A_2-A_1}_{\L1 ([0,t]\times\Omega;\reali)} \,.
  \end{align*}
\end{lemma}

\noindent The proof is a straightforward adaptation
from~\cite[Lemma~4.3]{teo}.

The $\tv$ bound obtained in the next lemma will be crucial in the
sequel.

\begin{lemma}
  \label{lem:tv}
  Let~\ref{it:omega}--\ref{item:H0}--\ref{it:H3} hold. Assume,
  moreover, that $A \in \L1 ([0,T]; \L\infty (\Omega;\reali))$ and for
  all $t \in [0,T]$, $A(t) \in \BV (\Omega;\reali)$. Let $c$
  satisfy~\ref{it:H1} and, moreover, $c (t) \in \C2 (\Omega;\reali^n)$
  for all $t \in [0,T]$ and
  $\grad \div c \in \L1 ([0,T]\times \Omega; \R^n)$. Then, the map $u$
  defined in~\eqref{eq:9} satisfies for all $t \in [0,T]$.
  \begin{align*}
    & \tv (u (t);\Omega)
    \\
    \le \
    &  \exp\left(
      \norma{A}_{\L1([0,t];\L\infty (\Omega; \R))}
      +
      \norma{D_x c}_{\L1([0,t];\L\infty (\Omega; \R^{n\times n}))}
      \right)
    \\
    & \times \Biggl(
      \tv(u_o) + \mathcal{O}(1) \norma{u_o}_{\L\infty(\Omega; \R)}
      + \int_0^t \tv\left( a (\tau)\right) \dd\tau
    \\
    & \qquad
      + \left(\norma{u_o}_{\L\infty (\Omega; \R)}
      + \norma{a}_{\L1([0,t];\L\infty(\Omega; \R))}\right)
      \int_0^t \left(\tv\left(A(\tau)\right)
      + \norma{\grad \div c(\tau) }_{\L1 (\Omega; \R^n)}\right) \dd{\tau}
      \Biggr).
  \end{align*}
\end{lemma}

\begin{proof}
  The proof extends that of~\cite[Lemma~4.4]{siam2018}, where a
  linear conservation law, i.e., with no source term, on a bounded
  domain is considered.

  We first regularize the initial datum $u_o$ and the functions $A$
  and $a$ appearing in the source term. In particular, we use the
  approximation of the initial datum constructed
  in~\cite[Lemma~4.3]{siam2018}, yielding a sequence
  $u_o^h \in \C3 (\Omega; \R)$ such that
  \begin{equation}
    \label{eq:27}
    \begin{aligned}
      \lim_{h \to +\infty} \norma{u_o^h - u_o}_{\L1 (\Omega;\R)} = \ &
      0, & u_o^h (\xi) = \ & 0 \text{ for all } \xi \in
      \partial\Omega,
      \\
      \norma{u_o^h}_{\L\infty (\Omega;\R)} \le \ &
      \norma{u_o}_{\L\infty (\Omega;\R)}, & \tv (u_o^h) \le \ &
      \mathcal{O}(1) \, \norma{u_o}_{\L\infty (\Omega;\R)} + \tv(u_o)
      \,.
    \end{aligned}
  \end{equation}
  Then, using~\cite[Formula~(1.8) and Theorem~1.17]{Giusti}, we
  regularize the functions $A$ and $a$ as follows. For all
  $t \in [0,T]$ and $h\in \N \setminus \{0\}$, there exist sequences
  $A_h(t), a_h(t) \in \C\infty (\Omega; \R)$ such that, for all
  $t \in [0,T]$,
  \begin{equation}
    \label{eq:29}
    \begin{aligned}
      \lim_{h\to+\infty} & \norma{A_h (t) - A(t)}_{\L1 (\Omega;\R)} =
      \ 0, & \lim_{h\to+\infty} &\norma{a_h (t) - a(t)}_{\L1
        (\Omega;\R)} = \ 0,
      \\
      &\norma{A_h (t)}_{\L\infty (\Omega;\R)} \le \
      \norma{A(t)}_{\L\infty (\Omega;\R)}, && \norma{a_h
        (t)}_{\L\infty (\Omega;\R)} \le \ \norma{a(t)}_{\L\infty
        (\Omega;\R)},
      \\
      \lim_{h\to+\infty} & \tv\left(A_h(t)\right) = \ \tv\left(A(t)
      \right), & \lim_{h\to+\infty} & \tv\left(a_h(t)\right) = \
      \tv\left(a(t) \right).
    \end{aligned}
  \end{equation}
  According to~\eqref{eq:9}, define the sequence $u_h$ corresponding
  to the sequences $u_o^h$, $A_h$, $a_h$, where the map $\mathcal{E}$
  in~\eqref{eq:13} is substituted by $\mathcal{E}_h$, defined
  accordingly exploiting $A_h$. By construction,
  $u_h (t)\in \C1 (\Omega; \R)$ for all $t \in [0,T]$, thus we can
  differentiate it. In particular, we are interested in the
  $\L1$--norm of $\nabla u_h(t)$. By~\eqref{eq:38}, the following
  decomposition holds:
  \begin{equation}
    \label{eq:17}
    \norma{\nabla u_h(t)}_{\L1 (\Omega; \R)}
    =
    \norma{\nabla u_h(t)}_{\L1 (X(t; 0, \Omega); \R)}
    +
    \norma{\nabla u_h(t)}_{\L1 (X(t; [0,t\mathclose[, \partial\Omega); \R)}.
  \end{equation}
  The two terms on the right hand side of~\eqref{eq:17} are treated
  separately. Focus on the first term: if $x \in X(t;0,\Omega)$,
  by~\eqref{eq:9}
  \begin{align*}
    & \grad u_h (t,x)
    \\= \
    & \mathcal{E}_h (0,t,x) \Bigl(
      \grad u_o^h\left(X(0;t,x)\right) D_x X (0;t,x)
    \\
    & + u_o^h\left( X (0;t,x)\right)
      \int_0^t \left(\grad A_h\left(s,  X (s;t,x)\right)
      - \grad \div c\left(s,  X (s;t,x)\right)\right) D_x X (s;t,x) \dd{s}
      \Bigr)
    \\
    & + \int_0^t \mathcal{E}_h(\tau,t,x) \Bigl(
      \grad a_h\left(\tau,X(\tau;t,x)\right) D_x X (\tau;t,x)
    \\
    & + a_h\left(\tau,X (\tau;t,x)\right)
      \int_\tau^t \left(\grad A_h\left(s,  X (s;t,x)\right)
      - \grad \div c\left(s,  X (s;t,x)\right)\right) D_x X (s;t,x) \dd{s}
      \Bigr) \dd\tau \,.
  \end{align*}
  Use the change of variables $y = X(0; t,x)$ in the first two lines
  above involving $u_o^h$, the change of variables $y = X(\tau; t,x)$
  in the latter two lines and the bound
  \begin{equation}
    \label{eq:36}
    \norma{D_x X (\tau;t,x)}
    \le
    \exp \left(\int_\tau^t
      \norma{D_x c(s)}_{\L\infty (\Omega; \R^{n\times n})} \dd{s}\right),
  \end{equation}
  that holds for every $t \in [0,T]$ by~\eqref{eq:pxoX}. We thus
  obtain
  \begin{align}
    \nonumber
    &  \norma{\grad u_h (t)}_{\L1(X(t; 0, \Omega);\R^n)} = \
      \int_{X(t; 0, \Omega)} \norma{\grad u_h (t,x)}\dd{x}
    \\ \label{eq:tv1}
    \le \
    & \exp\left(
      \int_0^t \left(\norma{A_h(\tau)}_{\L\infty (\Omega; \R)}
      + \norma{D_x c(\tau)}_{\L\infty (\Omega; \R^{n\times n})}\right) \dd{\tau}
      \right)
    \\\nonumber
    & \times \Biggl(
      \norma{\grad u_o^h}_{\L1 (\Omega;\R^n)}
      + \norma{u_o^h}_{\L\infty (\Omega; \R)}
      \int_0^t \left(\norma{\grad A_h(\tau)}_{\L1 (\Omega; \R^n)}
      + \norma{\grad \div c(\tau) }_{\L1 (\Omega; \R^n)}\right) \dd{\tau}
    \\\nonumber
    & \qquad
      + \int_0^t \norma{\grad a_h (\tau)}_{\L1 (X(\tau; 0, \Omega); \R^n)} \dd\tau
      + \int_0^t \norma{ a_h (\tau)}_{\L\infty (X(\tau; 0, \Omega); \R)}
    \\
    \nonumber
    & \qquad \qquad \qquad
      \times \left(\int_\tau^t \left(\norma{\grad A_h(s)}_{\L1 (X(s; 0, \Omega); \R^n)}
      + \norma{\grad \div c(s) }_{\L1 (X(s; 0, \Omega); \R^n)}\right)\dd{s}\right) \dd{\tau}
      \Biggr).
  \end{align}
  Pass to the second term on the right in~\eqref{eq:17}:
  if $x \in X (t; [0,t\mathclose[, \partial \Omega)$, by~\eqref{eq:9} and~\eqref{eq:36}
  \begin{align*}
    & \grad u_h (t,x)
    \\
    = \
    &  \int_{T(t,x)}^t \mathcal{E}_h(\tau,t,x) \Bigl(
      \grad a_h\left(\tau,X(\tau;t,x)\right) D_x X (\tau;t,x)
    \\
    & + a_h\left(\tau,X (\tau;t,x)\right)
      \int_\tau^t \left(\grad A_h\left(s,  X (s;t,x)\right)
      - \grad \div c\left(s,  X (s;t,x)\right)\right) D_x X (s;t,x) \dd{s}
      \Bigr) \dd\tau.
  \end{align*}
  For every $t \in [0,T]$, proceed similarly as above 
  using the change of variables $y=X(\tau;t,x)$:
  \begin{align}
    \nonumber
    & \norma{\grad u_h (t)}_{\L1( X (t; [0,t\mathclose[, \partial \Omega);\R)}
      =
      \int_{\Omega \setminus X(t; 0,\Omega)} \norma{\grad u_h (t,x)}\dd{x}
    \\
    \label{eq:tv2}
    \le\
    & \exp\left(
      \int_0^t \left(\norma{A_h(\tau)}_{\L\infty (\Omega; \R^n)}
      + \norma{D_x c(\tau)}_{\L\infty (\Omega; \R^{n\times n})}\right) \dd{\tau}
      \right)
    \\\nonumber
    & \times \Biggl(
      \int_0^t \int_{\Omega \setminus X (\tau;0,\Omega)}
      \norma{\grad a_h (\tau, y)}\dd{y}\dd{\tau}
      + \int_0^t
      \norma{ a_h (\tau)}_{\L\infty (\Omega \setminus X (\tau;0,\Omega); \R)} \dd\tau
    \\
    \nonumber
    & \qquad\times
      \left(
      \norma{\grad A_h}_{\L1 (\Omega \setminus X ([0,t];0,\Omega); \R^n)}
      +
      \norma{\grad \div c}_{\L1 (\Omega \setminus X ([0,t];0,\Omega); \R^n)}
      \right)
      \Biggr).
  \end{align}
  Inserting the estimates~\eqref{eq:tv1} and~\eqref{eq:tv2}
  into~\eqref{eq:17}, we thus obtain
  \begin{align*}
    & \norma{\grad u_h(t)}_{\L1 (\Omega;\R^n)}
    \\
    \le \
    & \exp
      \left(
      \norma{A_h}_{\L1([0,t];\L\infty (\Omega; \R^n))}
      +
      \norma{D_x c}_{\L1([0,t];\L\infty (\Omega; \R^{n\times n}))}
      \right)
    \\
    & \times \Biggl(
      \norma{\grad u_o^h}_{\L1 (\Omega;\R)}
      + \int_0^t \norma{\grad a_h (\tau)}_{\L1 (\Omega; \R^n)} \dd\tau
    \\
    & \quad + \left(\!\norma{u_o^h}_{\L\infty (\Omega; \R)}
      \!\! +\! \int_0^t \norma{ a_h (\tau)}_{\L\infty(\Omega; \R)}\dd\tau\!\right)
      \!\!
      \left(\norma{\grad A_h}_{\L1 ([0,t] \times \Omega; \R^n)}
      \!+\! \norma{\grad \div c }_{\L1 ([0,t] \times \Omega; \R^n)}\right)
      \!\!\Biggr).
  \end{align*}
  Since $u_h(t) \to u(t)$ in $\L1 (\Omega;\R)$, by the lower
  semicontinuity of the total variation and the
  hypotheses~\eqref{eq:27}--\eqref{eq:29} on the regularizing
  sequences $u_o^h$, $A_h$ and $a_h$, passing to the limit
  $h \to +\infty$, we complete the proof.
\end{proof}

It is on the basis of next Proposition that we give a definition of solution to
\eqref{eq:12}.

\begin{proposition}
  \label{prop:def}
  Let~\ref{it:omega} and~\ref{it:H1} hold.  Assume
  $u_o \in \L\infty (\Omega;\reali)$,
  $A \in \L\infty ([0,T] \times \Omega; \R)$ and
  $a \in \L1 \left([0,T]; \L\infty (\Omega; \R)\right)$. Then, the
  following statements are equivalent:
  \begin{enumerate}[label={\bf(\arabic*)}]
  \item\label{it:1} $u$ is defined by~\eqref{eq:9}, i.e., through
    integration along characteristics.

  \item\label{it:2} $u \in \L\infty ([0,T] \times \Omega; \R)$ is such
    that for any test function
    $\phi \in \Cc1 (\mathopen] -\infty, T\mathclose[ \times \Omega;
    \R)$,
    \begin{equation}
      \label{eq:41}
      \!\!\!\!\!\!\!\!\!\!\!\!\!\!\!\!\!\!\!\!\!\!\!\!
      \begin{array}{@{}r@{}}
        \displaystyle
       \int_0^T \int_\Omega
        \left(u (t,x)
        \left(\partial_t \phi (t,x) + c (t,x) \cdot \nabla \phi (t,x) \right)
        +
        \left(A (t,x) \, u(t,x) + a(t,x) \right)\phi (t,x)
        \right)
        \d{x}\d{t}
      \\
        \displaystyle
        + \int_\Omega u_o (x) \, \phi(0,x) \d{x}
        = 0.
      \end{array}
    \end{equation}

  \item\label{it:5} $u \in \L\infty ([0,T] \times \Omega; \R)$ is such
    that for any test function
    $\phi \in \W1\infty (\mathopen] -\infty, T\mathclose[ \times
    \Omega; \R)$, equality~\eqref{eq:41} holds.
  \end{enumerate}
\end{proposition}

\begin{proof}
  \paragraph{\ref{it:1}$\implies $\ref{it:2}}
  The proof exploits arguments similar to~\cite[Lemma~5.1]{CHM2011},
  see also~\cite[Lemma~2.7]{parahyp}.  Indeed, $u$ defined as
  in~\eqref{eq:9} is bounded by Item~\ref{item:4} in
  Lemma~\ref{lem:L1}.

  Let
  $\phi \in \Cc1 (\mathopen]-\infty, T\mathclose[ \times \Omega;
  \R)$. We prove that the equality~\eqref{eq:41} holds with $u$
  defined as in~\eqref{eq:9}. Notice that, for a fixed time
  $t \in [0,T]$, by~\eqref{eq:38} the domain $\Omega$ is contained in
  the disjoint union of $X(t; 0, \Omega)$ and
  $X(t; [0,t\mathclose[ , \partial\Omega)$. The first set accounts for
  the characteristics emanating from the initial datum, the second one
  for those coming from the boundary. Therefore, to prove that the
  integral equality~\eqref{eq:41} holds it is sufficient to verify
  that both the following integral equalities hold:
  \begin{align}
    \label{eq:datoIn}
    & \int_0^T \int_{X (t; 0, \Omega)}
      \left(u
      \left( \partial_t \phi + c \cdot \nabla \phi + A \, \phi \right)
      +
      a \, \phi
      \right)
      \d{x}\d{t}
      + \int_\Omega u_o (x) \, \phi(0,x) \d{x}
      = 0,
    \\
    \label{eq:bordo}
    & \int_0^T \int_{X(t; [0,t\mathclose[, \partial\Omega)}
      \left(u
      \left(\partial_t \phi + c \cdot \nabla \phi  + A \, \phi \right)
      +
      a \, \phi
      \right)
      \d{x}\d{t}
      = 0.
  \end{align}
  In order to prove~\eqref{eq:datoIn}, exploiting the change of
  variables $y = X(0;t,x)$, the first line in~\eqref{eq:9} can be
  rewritten for $x \in X (t,0,\Omega)$ as
  \begin{displaymath}
    u(t,x) =
    \left(u_o(y) + \mathcal{A}(t,y)\right) \frac{\mathscr{A}(t,y)}{J(t,y)}
    \quad \mbox{ where } \quad
    y = X(0;t,x)
  \end{displaymath}
  with
  \begin{align*}
    \mathscr{A} (t,y) = \
    & \exp\left(\int_0^t A\left(\tau, X (\tau; 0,y) \right)\dd{\tau}\right),
    &
      J (t,y) = \
    & \exp\left(\int_0^t \div c \left(\tau, X (\tau; 0,y) \right)\dd{\tau}\right),
    \\
    \mathcal{A}(t,y) = \
    & \int_0^t a\left(\tau,X (\tau;0,y)\right) \frac{J(\tau,y)}{\mathscr{A} (\tau,y)}\dd\tau.
  \end{align*}
  Therefore, the left hand side of~\eqref{eq:datoIn} now reads
  \begin{align*}
    \int_0^T\int_\Omega
    &
      \bigl[
      \left(u_o(y) + \mathcal{A}(t,y)\right) \frac{\mathscr{A}(t,y)}{J(t,y)}
      \left(
      \partial_t \phi \left(t,X(t;0,y)\right) \right.
    \\
    & \quad \left.
      + c \left(t,X(t;0,y)\right) \cdot \nabla \phi \left(t,X(t;0,y)\right) + A \left(t,X(t;0,y)\right) \phi \left(t,X(t;0,y)\right)
      \right) J(t,y)
    \\
    &
      + a\left(t,X(t;0,y)\right) \phi\left(t,X(t;0,y)\right) J(t,y)
      \bigr] \dd{y}\dd{t}
      + \int_\Omega u_o(x) \phi(0,x) \dd{x}
    \\
    = \
    &  \int_0^T\int_\Omega \frac{\d{}}{\d{t}}
      \left[ \left(
      u_o(y) + \mathcal{A}(t,y) \mathscr{A}(t,y)
      \right)
      \phi \left(t,X(t;0,y)\right)
      \right]\dd{y}\dd{t}
      + \int_\Omega u_o(x) \phi(0,x) \dd{x}
    \\
    =\
    & - \int_\Omega u_o(y) \phi(0,y) \dd{y}+ \int_\Omega u_o(x) \phi(0,x) \dd{x}
      - \int_\Omega \mathcal{A} (0,y) \mathscr{A}(0,y) \phi(0,y) \dd{y}
    \\
    = \
    & 0,
  \end{align*}
  since, for all $y \in \Omega$, $\phi(T,y) = 0 $ and, by definition,
  $\mathcal{A}(0,y) = 0$.

  Pass now to~\eqref{eq:bordo}. Here, for all $t \in [0,T]$, we use
  the change of variables
  \begin{displaymath}
    \begin{array}{ccc}
      \Omega_{\tau,x}^t
      &
        \to
      & \Omega_{\sigma,y}^t
      \\
      (\tau,x)
      & \mapsto
      & (\sigma,y)
    \end{array}
    \quad \mbox{ where }\quad
    \begin{array}{rcl}
      \Omega_{\tau,x}^t
      & =
      & \left\{
        (\tau,x) \colon \tau \in [T (t,x), t]
        \mbox{ and }
        x = X (t, [0,t\mathclose[,\partial\Omega)
        \right\}
      \\
      \Omega_{\sigma,y}^t
      & =
      & \left\{
        (\sigma,y) \colon \sigma \in [0, t]
        \mbox{ and }
        y = X (\sigma, [0,\sigma\mathclose[,\partial\Omega)
        \right\}
      \\
      \sigma
      & =
      & \tau
      \\
      y
      & =
      & X (\tau;t,x)
    \end{array}
  \end{displaymath}
  whose Jacobian, which also depends on $t$, is
  $\left. H (t,\sigma,y) \middle/ H (\sigma,\sigma,y) \right.$ where
  we set
  \begin{displaymath}
    H (t,\sigma,y)
    =
    \exp
    \int_0^t \div c \left(s,X (s;\sigma,y)\right) \d{s}
    \quad \mbox{ and }\quad
    \widehat{A} (t,\sigma,y)
    =
    \exp
    \int_0^t A\left(s, X (s; \sigma,y) \right)\d{\sigma} \,.
  \end{displaymath}
  Using~\eqref{eq:9}, we compute now the right hand side
  in~\eqref{eq:bordo} as follows:
  \begin{align*}
    &\int_0^T \int_{X(t; [0,t\mathclose[, \partial\Omega)}
      u
      \left(\partial_t \phi + c \cdot \nabla \phi  + A \, \phi \right) (t,x)
      \d{x}\d{t}
      +
      \int_0^T \int_{X(t; [0,t\mathclose[, \partial\Omega)}
      a \, \phi
      \d{x}\d{t}
    \\
    = \
    & \int_0^T \int_{X(t; [0,t\mathclose[, \partial\Omega)}
      \int_{T(t,x)}^t
      a\left(\tau,X(\tau;t,x)\right)\mathcal{E} (\tau,t,x) \d\tau
      \left(\partial_t \phi + c \cdot \nabla \phi  + A \, \phi \right)
      \d{x}\d{t}
    \\
    & \qquad +
      \int_0^T \int_{X(t; [0,t\mathclose[, \partial\Omega)}
      a (t,x) \, \phi (t,x)
      \d{x}\d{t}
    \\
    = \
    & \int_0^T
      \int_0^t
      \int_{X (\sigma,[0,\sigma\mathclose[,\partial\Omega)}
      a (\sigma,y) \,
      \dfrac{\widehat{A} (t,\sigma,y)}{\widehat{A} (\sigma,\sigma,y)}
    \\
    & \qquad \qquad \qquad \times
      \left(
      \frac{\d{\phi\left(t, X (t;\sigma,y) \right)}}{\d{t}}
      + A\left(t, X (t;\sigma,y)\right) \, \phi\left(t, X (t;\sigma,y)\right)
      \right)
      \d{y} \d{\sigma} \d{t}
    \\
    & \qquad +
      \int_0^T \int_{X(t; [0,t\mathclose[, \partial\Omega)}
      a (t,x) \, \phi (t,x)
      \d{x}\d{t}
    \\
    = \
    & \int_0^T
      \dfrac{\d~}{\d{t}}
      \left(
      \int_0^t
      \int_{X (\sigma,[0,\sigma\mathclose[,\partial\Omega)}
      a (\sigma,y) \,
      \dfrac{\widehat{A} (t,\sigma,y)}{\widehat{A} (\sigma,\sigma,y)} \,
      \phi\left(t, X (t;\sigma,y) \right)
      \d{y} \d{\sigma}\right) \d{t}
    \\
    = \
    & 0,
  \end{align*}
  since $\phi(T, \cdot) \equiv 0$.

\paragraph{\ref{it:2}$\implies $\ref{it:5}}
Fix
$\phi \in \W1\infty (\mathopen]-\infty,T \mathclose[ \times \Omega;
\reali)$. A standard construction, see~\cite[\S~1.14]{Giusti}, ensures
the existence of a sequence of functions
$\phi_h \in \Cc\infty(\reali^{n+1}; \reali_+)$ such that
$\phi_h \underset{h\to+\infty}{\longrightarrow} \phi$,
$\partial_t \phi_h \underset{h\to+\infty}{\longrightarrow} \partial_t
\phi$ and
$\nabla_x \phi_h \underset{h\to+\infty}{\longrightarrow} \nabla_x
\phi$ in
$\Lloc1 (\mathopen]-\infty,T \mathclose[ \times \Omega; \reali)$ and
$\Lloc1 (\mathopen]-\infty,T \mathclose[ \times \Omega; \reali^n)$.
Call $\chi_h$ a function in $\Cc\infty (\reali^n; \reali)$ such that
$\chi_h (x) = 1$ for all $x \in \Omega$ such that
$B(x,1/h) \subseteq \Omega$ and
$\norma{\nabla_x \chi_h} \leq 2\sqrt{n}\, h$ for all $x \in \reali^n$.

Then, we have
$\phi_h \, \chi_h \in \Cc1 (\mathopen]-\infty, T \mathclose[ \times
\Omega; \reali)$. Moreover,
$\phi_h \, \chi_h \underset{h\to+\infty}{\longrightarrow} \phi$ and
$\partial_t (\phi_h \, \chi_h) \underset{h\to+\infty}{\longrightarrow}
\partial_t \phi$ in
$\Lloc1 (\mathopen]-\infty,T \mathclose[ \times \Omega;
\reali)$. Concerning the space gradient, we have
\begin{displaymath}
  \nabla_x (\phi_h \, \chi_h)
  =
  \nabla_x \phi_h \; \chi_h + \phi_h \; \nabla_x \chi_h
  \quad \mbox{ and }
  \begin{array}{r@{\,}c@{\,}l@{\qquad}l}
    \nabla_x \phi_h \; \chi_h
    & \underset{h\to+\infty}{\longrightarrow}
    & \nabla_x \phi
    & \mbox{in } \Lloc1 (\mathopen]-\infty,T \mathclose[ \times \Omega; \reali) \,;
    \\
    \phi_h \; \nabla_x \chi_h
    & \underset{h\to+\infty}{\longrightarrow}
    & \phi
    & \mbox{a.e. in } \mathopen]-\infty,T \mathclose[ \times \Omega \,.
  \end{array}
\end{displaymath}
Therefore, for all $h$ by~\ref{it:2} we have
\begin{align*}
  0
  = \
  & \int_0^T \int_\Omega
    \left(u
    \left(\partial_t (\phi_h \, \chi_h) + c \cdot \nabla (\phi_h \, \omega_h)  \right)
    +
    \left(A \, u(t,x) + a(t,x) \right)(\phi_h \, \chi_h)
    \right)
    \d{x}\d{t}
  \\
  & + \int_\Omega u_o (x) \, (\phi_h \, \chi_h) (0,x) \d{x}
\end{align*}
and, by the Dominated Convergence Theorem, \ref{it:5} follows.

\paragraph{\ref{it:5}$\implies $\ref{it:1}} Inspired
by~\cite[Lemma~5.1]{CHM2011}, we first consider the case
$A \in (\C1 \cap \W1\infty) ([0,T]\times\Omega;\reali)$. Assume $u$
satisfies~\ref{it:5} and call $u_*$ the function defined
in~\eqref{eq:9}. Then, by the above implications
\ref{it:1}$\Rightarrow$\ref{it:2}$\Rightarrow$\ref{it:5}, the
difference $U = u-u_*$ satisfies for all test functions
$\tilde\phi \in \W1\infty (\mathopen]-\infty, T\mathclose[ \times
\Omega;\reali)$ the integral equality
\begin{equation}
  \label{eq:25}
  \int_0^T \int_\Omega
  U
  \left(\partial_t \tilde\phi + c  \cdot \nabla \tilde\phi
    +
    A \, \tilde\phi
  \right)
  \d{x}\d{t}
  = 0 \,.
\end{equation}
Proceed now exactly as in~\cite[Lemma~5.1]{CHM2011}, choosing
$\tau \in \mathopen]0, T\mathclose]$, a sequence
$\chi_h \in \Cc1 (\reali; \reali_+)$ with $\chi_h (t) = 1$ for all
$t \in [1/h, \tau-1/h]$ and $\modulo{\chi'_h} \leq 2h$. If
$\phi \in \W1\infty (\mathopen]-\infty, T\mathclose[ \times
\Omega;\reali)$ then,
$(\phi \, \chi_h) \in \W1\infty (\mathopen]-\infty, T\mathclose[
\times \Omega;\reali)$. Choosing $\phi \, \chi_h$ as $\tilde \phi$
in~\eqref{eq:25}, and passing to the limit $h \to +\infty$ via the
Dominated Convergence Theorem, we get
\begin{equation}
  \label{eq:26}
  \int_0^\tau \int_\Omega
  U
  \left(
    \partial_t \phi + c  \cdot \nabla \phi
    +
    A \, \phi
  \right)
  \d{x}\d{t}
  -
  \int_\Omega U (\tau,x) \, \phi (\tau,x) \d{x}
  = 0 \,.
\end{equation}
Fix an arbitrary $\eta \in \Cc1 (\Omega;\reali)$ and let $\phi$ solve
\begin{displaymath}
  \left\{
    \begin{array}{l@{\qquad}r@{\,}c@{\,}l}
      \partial_t \phi + c \cdot \nabla \phi + A\, \phi = 0
      & (t,x)
      & \in
      & \Omega
      \\
      \phi (t,\xi) = 0
      & (t,\xi)
      & \in
      & \partial\Omega
      \\
      \phi (\tau,x) = \eta (x)
      & (\tau,x)
      & \in
      & \Omega\,.
    \end{array}
  \right.
\end{displaymath}
Note that $\phi$ can be computed through integration along (backward)
characteristics and hence
$\phi \in \W1\infty (\mathopen]-\infty, T\mathclose[ \times \Omega;\reali)$.
With this choice, \eqref{eq:26} yields
$\int_\Omega U (\tau,x) \, \eta (x) \d{x} = 0$ for all
$\eta \in \Cc1 (\Omega;\reali)$, so that $U (\tau,x) =0$ for all
$x \in \Omega$. By the arbitrariness of $\tau$, we have $U \equiv 0$,
hence $u = u_*$.

Let now $A \in \L\infty ([0,T] \times \Omega;\reali)$, call $u_*$ the
function constructed in~\eqref{eq:9} and assume there is a function
$u$ satisfying~\ref{it:5}.  Construct a sequence
$A_h \in (\C1 \cap \W1\infty) ([0,T]\times\Omega;\reali)$ such that
$A_h \underset{h\to+\infty}{\longrightarrow} A$ in
$\L1 ([0,T]\times\Omega;\reali)$. Call $u_h$ the function constructed
as in~\eqref{eq:9} with $A_h$ in place of $A$. For any $t \in [0,T]$,
we have $u_h (t) \underset{h\to+\infty}{\longrightarrow} u_* (t)$ in
$\L1 (\Omega; \reali)$, by Lemma~\ref{lem:A}. Moreover, for all
$\phi \in \W1\infty (\mathopen]-\infty,T\mathclose[ \times
\Omega;\reali)$,
\begin{align*}
  0
  = \
  & \int_0^T \int_\Omega
    \left(u
    \left(\partial_t \phi  + c  \cdot \nabla \phi  \right)
    +
    \left(A  \, u + a \right)\phi
    \right)
    \d{x}\d{t}
    + \int_\Omega u_o (x) \, \phi(0,x) \d{x}
  \\
  &-
    \int_0^T \int_\Omega
    \left(u_h
    \left(\partial_t \phi  + c  \cdot \nabla \phi  \right)
    +
    \left(A_h  \, u_h + a \right)\phi
    \right)
    \d{x}\d{t}
    - \int_\Omega u_o (x) \, \phi(0,x) \d{x}
  \\
  = \
  & \int_0^T \int_\Omega
    (u - u_h)
    \left(\partial_t \phi  + c  \cdot \nabla \phi +A_h \, \phi \right)
    \d{x}\d{t}
    +
    \int_0^T \int_\Omega (A- A_h) \, u \, \phi \d{x}\d{t}.
\end{align*}
The latter summand vanishes since $A_h \to A$ in $\L1$. The former
summand, thanks to the regularity of $A_h$, can be treated by the
procedure above, obtaining, for all $\eta \in \Cc1 (\Omega;\reali)$
and for a sequence of real numbers $\epsilon_h$ converging to $0$,
\begin{displaymath}
  0
  =
  \int_\Omega
  \left(u (\tau,x) - u_h (\tau,x) \right)
  \eta (x)
  \d{x}
  + \epsilon_h \,.
\end{displaymath}
The above relation ensures that $u_h (\tau,x) \to u (\tau,x)$ for
a.e.~$x \in \Omega$ as $h \to +\infty$. Therefore, for all
$t \in \mathopen]0,T\mathclose]$,
\begin{displaymath}
  \norma{u_* (t) - u (t)}_{\L1 (\Omega;\reali)}
  \leq
  \norma{u_* (t) - u_h (t)}_{\L1 (\Omega;\reali)}
  +
  \norma{u_h (t) - u (t)}_{\L1 (\Omega;\reali)}
  \underset{h\to+\infty}{\longrightarrow}
  0 \,,
\end{displaymath}
completing the proof.
\end{proof}

\begin{definition}
  \label{def:Hyp}
  A map $u \in \L\infty ([0,T]\times\Omega; \reali)$ is a solution
  to~\eqref{eq:12} if it satisfies any of the requirements~\ref{it:1},
  \ref{it:2} or~\ref{it:5} in~\Cref{prop:def}.
\end{definition}

By techniques similar to those in~\cite{parahyp}, one can verify that
a solution to~\eqref{eq:12} in the sense of~\Cref{def:Hyp} is a
\emph{weak entropy solution} also in any of the
senses~\cite{MR2322018, MR1884231}, or~\cite{MR542510} in the $\BV$
case, see~\cite{MR3819847} for a comparison. Here, as is well known,
the linearity of the convective part in~\eqref{eq:12} allows to avoid
the introduction of any entropy condition, as also remarked
in~\cite{KPS2018bounded}.

\begin{lemma}
  \label{lem:c}
  Let~\ref{it:omega}--\ref{item:H0}--\ref{it:H2}--\ref{it:H3}
  hold. Fix $c_1, c_2$ satisfying~\ref{it:H1} and moreover, for
  $i=1,2$, $c_i (t) \in \C2 (\Omega;\reali^n)$ for all $t \in [0,T]$
  and $\grad \div c_i \in \L1 ([0,T]\times \Omega; \R^n)$.  Then, the
  maps $u_1,u_2$ defined in~\eqref{eq:9} satisfy
  \begin{displaymath}
    \norma{u_2 (t) - u_1 (t)}_{\L1 (\Omega;\reali)}
    \le
    \O
    \left(
      \norma{c_1 - c_2}_{\L1 ([0,t]; \L\infty (\Omega;\reali^n))}
      +
      \norma{\div(c_1 - c_2)}_{\L1 ([0,t]; \L\infty (\Omega;\reali))}
    \right)
  \end{displaymath}
  and a precise expression for the constant $\O$ is provided in the
  proof.
\end{lemma}

\begin{proof}
  Following the proof of~\Cref{lem:tv}, we first regularize the
  initial datum $u_o$ as in~\eqref{eq:27} and the functions $A$ and
  $a$ appearing in the source term as in~\eqref{eq:29}, obtaining
  $u_1^h$ and $u_2^h$ by~\eqref{eq:9}. By~\Cref{prop:def}, the
  difference $u_2^h - u_1^h$ solves
  \begin{displaymath}
    \left\{
      \begin{array}{l}
        \partial_t (u_2^h - u_1^h)
        +
        \div \left(c_2 (u_2^h-u_1^h)\right)
        =
        A_h (u_2^h-u_h) + \alpha_h
        \\
        (u_2^h - u_1^h) (0) = 0
      \end{array}
    \right.
    \mbox{ where }
    \alpha_h = - \div\left((c_2 - c_1) u_1^h\right)
  \end{displaymath}
  in the sense of~\Cref{def:Hyp}.  Apply~\Cref{item:3}
  in~\Cref{lem:L1} to get
  \begin{equation}
    \label{eq:20}
    \norma{u_2^h(t) - u_1^h(t)}_{\L1 (\Omega; \R)}
    \leq
    \norma{\div\left((c_2 - c_1) u_1^h\right)}_{\L1 ([0,t]\times \Omega; \R)}
    \exp\left(\norma{A}_{\L\infty ([0,t]\times\Omega; \R)}t\right),
  \end{equation}
  where we use the estimate
  $\norma{A_h(\tau)}_{\L\infty (\Omega; \R)} \le
  \norma{A(\tau)}_{\L\infty (\Omega; \R)}$ for all $\tau \in [0,T]$.
  Observe that
  \begin{align*}
    & \norma{\div\left(\left(c_2 (\tau)  - c_1 (\tau) \right)u_1^h (\tau) \right)}_{\L1 (\Omega;\R)}
    \\
    \le \
    & \norma{u_1^h (\tau)}_{\L1 (\Omega;\R)}
      \norma{\div\left(c_2 (\tau)  - c_1 (\tau)\right)}_{\L\infty (\Omega;\R)}
      +
      \norma{c_2 (\tau) -c_1 (\tau)}_{\L\infty (\Omega; \R^n)}
      \norma{\nabla u_1^h (\tau)}_{\L1 (\Omega; \R)}
    \\
    \le \
    & \left( \norma{u_o^h}_{\L1 (\Omega; \R)}
      + \norma{a_h}_{\L1([0,\tau] \times \Omega; \R)}\right)
      \exp\left(\norma{A_h}_{\L\infty ([0,\tau] \times \Omega; \R)} t\right)
      \norma{\div\left(c_2 (\tau)  - c_1 (\tau)\right)}_{\L\infty (\Omega;\R)}
    \\
    & + \norma{c_2 (\tau) -c_1 (\tau)}_{\L\infty (\Omega; \R^n)}
      \exp
      \left(
      \norma{A_h}_{\L1([0,\tau];\L\infty (\Omega; \R^n))}
      +
      \norma{D_x c_1}_{\L1([0,\tau];\L\infty (\Omega; \R^{n\times n}))}
      \right)
    \\
    & \times \Biggl(
      \norma{\grad u_o^h}_{\L1 (\Omega;\R)}
      + \int_0^\tau \norma{\grad a_h (s)}_{\L1 (\Omega; \R^n)} \dd{s}
    + \left(\norma{u_o^h}_{\L\infty (\Omega; \R)}
      \!+\! \int_0^\tau \norma{ a_h (s)}_{\L\infty(\Omega; \R)}\dd{s}\right)
    \\
    & \qquad\quad \times \left(\norma{\grad A_h}_{\L1 ([0,\tau] \times \Omega; \R^n)}
      + \norma{\grad \div c_1 }_{\L1 ([0,\tau] \times \Omega; \R^n)}\right)
      \Biggr)
    \\
    \le \
    & \left( \norma{u_o^h}_{\L1 (\Omega; \R)}
      + \norma{a_h}_{\L1([0,\tau] \times \Omega; \R)}\right)
      \exp\left(\norma{A}_{\L\infty ([0,\tau] \times \Omega; \R)} t\right)
      \norma{\div\left(c_2 (\tau)  - c_1 (\tau)\right)}_{\L\infty (\Omega;\R)}
    \\
    & + \norma{c_2 (\tau) -c_1 (\tau)}_{\L\infty (\Omega; \R^n)}
      \exp
      \left(
      \norma{A}_{\L1([0,\tau];\L\infty (\Omega; \R^n))}
      +
      \norma{D_x c_1}_{\L1([0,\tau];\L\infty (\Omega; \R^{n\times n}))}
      \right)
    \\
    & \times \Biggl(
      \tv(u_o) + \mathcal{O}(1)\norma{u_o}_{\L\infty (\Omega;\R)}
      + \int_0^\tau \norma{\grad a_h (s)}_{\L1 (\Omega; \R^n)} \dd{s}
    \\
    & \quad \!+\! \left(\norma{u_o}_{\L\infty (\Omega; \R)}
      \!+\! \int_0^\tau \norma{ a (s)}_{\L\infty(\Omega; \R)}\dd{s}\right)\!\!
      \left(\norma{\grad A_h}_{\L1 ([0,\tau] \times \Omega; \R^n)}
      \!+\!
      \norma{\grad \div c_1 }_{\L1 ([0,\tau] \times \Omega; \R^n)}\right)\!\!
      \Biggr),
  \end{align*}
  where we used~\Cref{item:3} in~\Cref{lem:L1}, \Cref{lem:tv} and the
  hypotheses~\eqref{eq:27}--\eqref{eq:29} on the regularizing
  sequences $u_o^h$, $A_h$ and $a_h$.  By the triangular inequality
  and the above computations,
  \begin{align}
    \nonumber
    & \norma{u_2(t) - u_1(t)}_{\L1 (\Omega; \R)}
    \\ \nonumber
    \le\
    &
      \norma{u_2(t) - u_2^h(t)}_{\L1 (\Omega; \R)}
      +
      \norma{u_2^h(t) - u_1^h(t)}_{\L1 (\Omega; \R)}
      +
      \norma{u_1(t) - u_1^h(t)}_{\L1 (\Omega; \R)}
    \\ \label{eq:21}
    \le\
    & \norma{u_2(t) - u_2^h(t)}_{\L1 (\Omega; \R)}
    \\\label{eq:22}
    & \begin{aligned}
      & + \left( \norma{u_o^h}_{\L1 (\Omega; \R)}
        + \norma{a_h}_{\L1([0,t] \times \Omega; \R)}\right)
      \\
      & \quad \times \exp\left(\norma{A}_{\L\infty ([0,t] \times \Omega; \R)} t\right)
      \int_0^t \norma{\div\left(c_2 (\tau)  - c_1 (\tau)\right)}_{\L\infty (\Omega;\R)} \dd\tau
    \end{aligned}
    \\ \nonumber
    & + \int_0^t \norma{c_2 (\tau) -c_1 (\tau)}_{\L\infty (\Omega; \R^n)} \dd\tau
      \,
      \exp
      \left(
      \norma{A}_{\L1([0,t];\L\infty (\Omega; \R^n))}
      +
      \norma{D_x c_1}_{\L1([0,t];\L\infty (\Omega; \R^{n\times n}))}
      \right)
    \\\label{eq:23}
    & \quad \times \Biggl(
      \tv(u_o) + \mathcal{O}(1)\norma{u_o}_{\L\infty (\Omega;\R)}
      + \int_0^t \norma{\grad a_h (s)}_{\L1 (\Omega; \R^n)} \dd{s}
    \\\nonumber
    & \quad\, +\! \left(\norma{u_o}_{\L\infty (\Omega; \R)}
      \!+\! \int_0^t\! \norma{ a (s)}_{\L\infty(\Omega; \R)}\dd{s}\right)\!\!
      \left(\norma{\grad A_h}_{\L1 ([0,t] \times \Omega; \R^n)}
      \!+\! \norma{\grad \div c_1 }_{\L1 ([0,t] \times \Omega; \R^n)}\right)\!\!
      \Biggr)
    \\\label{eq:24}
    & +
      \norma{u_1(t) - u_1^h(t)}_{\L1 (\Omega; \R)},
  \end{align}
  and in the limit $h \to +\infty$ we treat each term separately. By
  construction, \eqref{eq:21} and~\eqref{eq:24} converge to zero as
  $h \to + \infty$. By the hypotheses~\eqref{eq:27} on $u_o^h$
  and~\eqref{eq:29} on $A_h$ and $a_h$, in the limit we thus obtain
  \begin{align*}
    & \norma{u_2(t) - u_1(t)}_{\L1 (\Omega; \R)}
    \\
    \le \
    & \left( \norma{u_o}_{\L1 (\Omega; \R)}
      + \norma{a}_{\L1([0,t] \times \Omega; \R)}\right)
      \exp\left(\norma{A}_{\L\infty ([0,t] \times \Omega; \R)} t\right)
      \int_0^t
      \norma{\div\left(c_2 (\tau)  - c_1 (\tau)\right)}_{\L\infty (\Omega;\R)} \dd\tau
    \\
    & + \int_0^t \norma{c_2 (\tau) -c_1 (\tau)}_{\L\infty (\Omega; \R^n)} \dd\tau
      \,
      \exp
      \left(
      \norma{A}_{\L1([0,t];\L\infty (\Omega; \R^n))}
      +
      \norma{D_x c_1}_{\L1([0,t];\L\infty (\Omega; \R^{n\times n}))}
      \right)
    \\
    & \times \Biggl(
      \tv(u_o) + \mathcal{O}(1)\norma{u_o}_{\L\infty (\Omega;\R)}
      + \int_0^t \tv\left( a (s)\right) \dd{s}
    \\
    & \quad + \left(\norma{u_o}_{\L\infty (\Omega; \R)}
      \!+\! \int_0^t \norma{ a (s)}_{\L\infty(\Omega; \R)}\dd{s}\right)\!\!
      \left( \int_0^t\tv\left( A (s)\right)\dd{s}
      + \norma{\grad \div c_1 }_{\L1 ([0,t] \times \Omega; \R^n)}\right)\!\!
      \Biggr),
  \end{align*}
  concluding the proof.
\end{proof}

\begin{lemma}
  \label{lem:positive}
  Let~\ref{it:omega}--\ref{it:H1} hold. Assume moreover that
  $u_o \in \L\infty (\Omega;\reali)$, with $u_o\geq 0$,
  $A \in \L\infty ([0,T]\times\Omega; \reali)$ and
  $a \in \L1 \left([0,T]; \L\infty(\Omega; \reali) \right)$, with
  $a \geq 0$. Then, the solution $u$ is positive.
\end{lemma}

\noindent The proof is an immediate consequence of the
representation~\eqref{eq:9}.

\begin{lemma}
  \label{lem:tcont}
  Let~\ref{item:H0}--\ref{it:H1}--\ref{it:H2}--\ref{it:H3}
  hold. Assume, moreover, that $c (t) \in \C2 (\Omega;\reali^n)$ for
  all $t \in [0,T]$ and
  $\grad \div c \in \L1 ([0,T]\times \Omega; \R^n)$. If
  $u \in \L\infty ([0,T]\times \Omega; \reali )$ is as
  in~\eqref{eq:9}, then $u$ is $\L1$--Lipschitz continuous in time:
  for all $t_1,t_2 \in [0,T]$, with $t_1< t_2$
  \begin{equation}
    \label{eq:42}
    \norma{u (t_2) - u (t_1)}_{\L1 (\Omega;\reali)}
    \leq
    \O (t_2 - t_1)
  \end{equation}
  where $\O$ depends on norms of $c,A,a$ on the interval $[0, t_2]$
  and of $u_o$.
\end{lemma}

\begin{proof}
  By~\eqref{eq:38}, the following decomposition holds
  \begin{equation}
    \label{eq:45}
    \begin{array}{rcl}
      \norma{u (t_2) - u (t_1)}_{\L1 (\Omega;\reali)}
      & =
      & \norma{u (t_2) - u (t_1)}_{\L1 (X(t_2;t_1,\Omega);\reali)}
      \\
      &
      & + \norma{u (t_2) - u (t_1)}_{\L1 (X(t_2;[t_1,t_2[,\partial\Omega);\reali)} \,.
    \end{array}
  \end{equation}
  Estimate the two latter summands in~\eqref{eq:45}
  separately. By~\eqref{eq:9}
  \begin{align}
    \nonumber
    & \norma{u (t_2) - u (t_1)}_{\L1 (X(t_2;t_1,\Omega);\reali)}
    \\
    \nonumber
    \le\
    & \int_{X(t_2;t_1,\Omega)} \modulo{
      u\left(t_1,X(t_1;t_2,x)\right) \, \mathcal{E} (t_1,t_2,x) - u (t_1,x)
      } \d{x}
    \\
    \nonumber
    & \qquad +
      \int_{X(t_2;t_1,\Omega)} \int_{t_1}^{t_2}
      \modulo{a\left(\tau,X (\tau;t_2,x)\right) \, \mathcal{E} (\tau,t_2,x)}
      \d\tau \d{x}
    \\
    \label{eq:46}
    \le \
    & \int_{X(t_2;t_1,\Omega)} \modulo{
      u\left(t_1,X(t_1;t_2,x)\right) - u (t_1,x)
      } \, \mathcal{E} (t_1,t_2,x) \d{x}
    \\
    \label{eq:47}
    & + \int_{X(t_2;t_1,\Omega)} \modulo{ u (t_1,x) } \;
      \modulo{\mathcal{E} (t_1,t_2,x)-1}\d{x}
    \\
    \label{eq:48}
    & +
      \int_{X(t_2;t_1,\Omega)} \int_{t_1}^{t_2}
      \modulo{a\left(\tau,X (\tau;t_2,x)\right) \, \mathcal{E} (\tau,t_2,x)}
      \d\tau \d{x} \,.
  \end{align}
  To estimate~\eqref{eq:46}, we use~\cite[Lemma~5.1]{MauroMatthew} so
  that we obtain
  \begin{align*}
    & \int_{X(t_2;t_1,\Omega)} \modulo{
      u\left(t_1,X(t_1;t_2,x)\right) - u (t_1,x)
      } \, \mathcal{E} (t_1,t_2,x) \d{x}
    \\
    \le \
    & \dfrac{\norma{c}_{\L\infty ([t_1,t_2]\times\Omega;\reali^n)}}{\norma{D_x c}_{\L\infty ([t_1,t_2]\times\Omega;\reali^{n\times n})}}
      \left(
      e^{\norma{D_x c}_{\L\infty ([t_1,t_2]\times\Omega;\reali^{n\times n})} (t_2-t_1)} -1\right) \tv\left(u (t_1)\right)
    \\
    \le \
    & \norma{c}_{\L\infty ([t_1,t_2]\times\Omega;\reali^n)}
      e^{\norma{D_x c}_{\L\infty ([t_1,t_2]\times\Omega;\reali^{n\times n})} (t_2-t_1)}
      \tv\left(u (t_1)\right)
      (t_2 - t_1),
  \end{align*}
  and the total variation of $u$ might be estimated thanks
  to~\Cref{lem:tv}. The bounds for~\eqref{eq:47} and~\eqref{eq:48} follow from the
  definition~\eqref{eq:13} of $\mathcal{E}$:
  \begin{align*}
    & \int_{X(t_2;t_1,\Omega)} \modulo{ u (t_1,x) } \;
      \modulo{\mathcal{E} (t_1,t_2,x)-1}\d{x}
    \\
    \le
    \
    & \norm{u(t_1)}_{\L1 (\Omega; \R)} (t_2-t_1) \left(
      \norm{A}_{\L\infty ([t_1,t_2]\times\Omega; \R)}
      +\norm{\div c}_{\L\infty ([t_1,t_2]\times\Omega; \R)}
      \right)
    \\
    & \times \exp\left(\left(
      \norm{A}_{\L\infty ([t_1,t_2]\times\Omega; \R)}
      +\norm{\div c}_{\L\infty ([t_1,t_2]\times\Omega; \R)}
      \right) (t_2-t_1)\right) \,;
    \\
    & \int_{X(t_2;t_1,\Omega)} \int_{t_1}^{t_2}
      \modulo{a\left(\tau,X (\tau;t_2,x)\right) \, \mathcal{E} (\tau,t_2,x)}
      \d\tau \d{x}
    \\
    \le \
    & (t_2-t_1) \, \norm{a}_{\L\infty ([t_1,t_2]; \L1 (\Omega; \R))}
      \exp\left(\int_{t_1}^{t_2}\norm{A (\tau)}_{\L\infty (\Omega; \R)}
      + \norm{\div c}_{\L\infty (\Omega; \R)}\d\tau\right) \,.
  \end{align*}

  Consider now the second summand in~\eqref{eq:45}. Introduce
  $T_{t_1}(t_2,x) = \inf \{s \in [t_1,t_2] \colon$ $X(s;t_2,x) \in
  \Omega\}$ and compute
  \begin{align*}
    & \norma{u (t_2) - u (t_1)}_{\L1 (X(t_2;[t_1,t_2[,\partial\Omega);\reali)}
    \\
    \le \
    & \int_{X(t_2;[t_1,t_2[,\partial\Omega)}
      \modulo{
      \int_{T_{t_1} (t_2,x)}^{t_2}
      a\left(\tau,X (\tau;t_2,x)\right) \, \mathcal{E} (\tau,t_2,x)
      \d\tau
      }
      \d{x}.
  \end{align*}
  The same procedure used to bound~\eqref{eq:48} applies, completing
  the proof.
\end{proof}

\subsection{Coupling}
\label{sec:coupling}

\begin{proofof}{Theorem~\ref{thm:main}}
  Fix $T>0$.  Define $u_0 (t,x) = u_o (x)$ and $w_0(t,x) = w_o (x)$
  for all $(t,x) \in [0,T] \times \Omega$. For $i \in \naturali$,
  define recursively $u_{i+1}$ and $w_{i+1}$ as solutions to
  \begin{equation}
    \label{eq:28}
    \left\{
      \begin{array}{l@{\qquad}r@{\,}c@{\,}l}
        \partial_t u_{i+1}
        + \div \left(u_{i+1} \, c_i(t,x) \right)
        =
        A_i (t,x)\, u_{i+1} + a(t,x)
        & (t,x)
        & \in
        & [0,T] \times \Omega
        \\
        u (t,\xi) = 0
        & (t,\xi)
        & \in
        & [0,T] \times \partial \Omega
        \\
        u (0,x) = u_o (x)
        & x
        & \in
        & \Omega
      \end{array}
    \right.
  \end{equation}
  \begin{equation}
    \label{eq:30}
    \left\{
      \begin{array}{l@{\qquad}r@{\,}c@{\,}l}
        \partial_t w_{i+1}
        - \mu \, \Delta w_{i+1}
        =
        B_i (t,x) \, w_{i+1} + b (t,x)
        & (t,x)
        & \in
        & [0,T] \times \Omega
        \\
        w (t,\xi) = 0
        & (t,\xi)
        & \in
        & [0,T] \times \partial \Omega
        \\
        w (0,x) = w_o (x)
        & x
        & \in
        & \Omega
      \end{array}
    \right.
  \end{equation}
  where
  \begin{equation}
    \label{eq:31}
    \begin{array}{rcl}
      c_i (t,x)
      & =
      & v (t,w_i) (x)
      \\
      A_i (t,x)
      & =
      & \alpha \left(t,x,w_i (t,x)\right)
    \end{array}
    \qquad\qquad
    B_i (t,x)
    =
    \beta \left(t,x,u_i (t,x),w_i (t,x)\right) \,.
  \end{equation}
  We aim to prove that $(u_i,w_i)$ is a Cauchy sequence with respect
  to the $\L\infty ([0,T] ; \L1 (\Omega;\reali^2))$ distance as soon
  as $T$ is sufficiently small.

  Observe first that problem~\eqref{eq:30} fits into the framework
  of~\Cref{sec:parabolicDirichlet}, while problem~\eqref{eq:28} fits
  into the framework of~\Cref{sec:hyperbolic}.

  Consider the $w$ component. \Cref{prop:superStimePara} applies,
  ensuring the existence of a solution to~\eqref{eq:30} for all
  $i\in\N$. Moreover, if $b \geq 0$ and the initial datum $w_o$ is
  positive, the solution $w_i$ is positive.  By~\ref{it:betaFinal}
  and~\eqref{eq:31}, for all $i\in \N$, $B_i$ satisfies~\ref{it:P2}
  and for all $\tau \in [0,T]$
  \begin{equation}
    \label{eq:B_Linf}
    \norma{B_{i} (\tau)}_{\L\infty (\Omega; \R)}
    \le
    K_\beta \,,
  \end{equation}
  while by~\ref{it:b} the function $b$ satisfies~\ref{it:P3}.  The
  following uniform bounds on $w_i$ hold for every $i \in \N$:
  by~\ref{it:P_apriori} and~\ref{it:P_tv}
  in~\Cref{prop:superStimePara}, exploiting also~\eqref{eq:B_Linf},
  for all $\tau \in [0,T]$,
  \begin{align}
    \label{eq:w_L1}
    \norma{w_i (\tau)}_{\L1 (\Omega; \R)}
    \le \
    & e^{K_\beta \, \tau} \left(
      \norma{w_o}_{\L1 (\Omega;\reali)}
      +
      \norma{b}_{\L1 ([0,\tau]\times\Omega;\reali)}
      \right)
      = \colon
      C_{w,1}(\tau),
    \\
    \label{eq:w_Linf}
    \norma{w_i (\tau)}_{\L\infty (\Omega;\reali)}
    \le \
    & e^{K_\beta \, \tau}  \left(
      \norma{w_o}_{\L\infty (\Omega;\reali)}
      +
      \norma{b}_{\L1 ([0,\tau]; \L\infty(\Omega;\reali))}
      \right)
      = \colon
      C_{w,\infty} (\tau),
    \\
    \label{eq:w_TV}
    \tv\left(w_i (\tau, \cdot)\right)
    \le \
    & \tv (w_o) + \int_0^\tau \tv\left(b (s)\right)\dd{s}
      + \mathcal{O} (1) \sqrt{\tau}\, K_\beta \, \norma{w_i (\tau)}_{\L1 (\Omega; \R)}
      = \colon
      C_w^{\tv} (\tau).
  \end{align}
  By~\ref{it:P_stability} in~\Cref{prop:superStimePara} we get
  \begin{align}
    \nonumber
    \norma{w_{i+1} (t) - w_i (t)}_{\L1 (\Omega;\reali)}
    \leq \
    & \norma{B_i-B_{i-1}}_{\L1 ([0,t] \times \Omega; \R)}
      \left(\norma{w_o}_{\L\infty (\Omega; \R)}
      + \norma{b}_{\L1 ([0,t];\L\infty (\Omega; \R))}\right)
    \\
    \label{eq:33}
    & \times \exp\int_0^t\left(
      \norma{B_i (\tau)}_{\L\infty (\Omega; \R)}
      +
      \norma{B_{i-1} (\tau)}_{\L\infty (\Omega; \R)}
      \right)\d{\tau}.
  \end{align}
  By~\eqref{eq:31}, exploiting the hypothesis~\ref{it:betaFinal} we
  obtain
  \begin{align*}
    & \norma{B_i-B_{i-1}}_{\L1 ([0,t] \times \Omega; \R)}
    \\
    = \
    & \int_0^t\int_\Omega
      \modulo{\beta \left(\tau,x,u_i (\tau,x), w_i (\tau,x)\right)
      -\beta \left(\tau,x,u_{i-1} (\tau,x), w_{i-1} (\tau,x)\right)}
      \d{x}\d\tau
    \\
    \le \
    & K_\beta \left(
      \norma{u_i - u_{i-1}}_{\L1 ([0,t] \times \Omega; \R)}
      +
      \norma{w_i - w_{i-1}}_{\L1 ([0,t] \times \Omega; \R)}
      \right).
  \end{align*}
  Therefore, using also~\eqref{eq:B_Linf} and the notation introduced
  in~\eqref{eq:w_Linf}, \eqref{eq:33} becomes
  \begin{equation}
    \label{eq:diff_wi}
    \begin{aligned}
      \norma{w_{i+1} (t) - w_i (t)}_{\L1 (\Omega;\reali)} \leq \ &
      K_\beta \, e^{t \, K_\beta} C_{w, \infty} (t)
      \\
      & \times \left( \norma{u_i - u_{i-1}}_{\L1 ([0,t] \times \Omega;
          \R)} + \norma{w_i - w_{i-1}}_{\L1 ([0,t] \times \Omega; \R)}
      \right).
    \end{aligned}
  \end{equation}

  \smallskip

  Pass now to the $u$ component. The results of~\Cref{sec:hyperbolic}
  applies, ensuring the existence of a solution to~\eqref{eq:28} for
  all $i\in\N$. Moreover, if $a \geq 0$ and the initial datum $u_o$ is
  positive, the solution $u_i$ is positive, see~\Cref{lem:positive}.
  By~\ref{it:alpha} and~\eqref{eq:31}, for every $i\in \N$ we have
  that $A_i$ satisfies~\ref{it:H2} and for all $\tau \in [0,T]$,
  exploiting~\eqref{eq:w_Linf} and~\eqref{eq:w_TV},
  \begin{align}
    \label{eq:A_Linf}
    \norma{A_{i} (\tau)}_{\L\infty (\Omega; \R)}
    \le \
    & K_\alpha \left(1+C_{w,\infty} (\tau)\right),
    \\
    \nonumber
    \tv \left(A_i (\tau, \cdot)\right)
    = \
    & \tv \alpha\left(\tau, \cdot, w_i (\tau, \cdot)\right)
    \\ \nonumber
    \le \
    & K_\alpha  \left(1
      + C_{w, \infty} (\tau)
      + \tv \left(w_i (\tau,\cdot)\right)\right)
    \\ \label{eq:A_tv}
    \le \
    &  K_\alpha \left(1
      + C_{w, \infty} (\tau)
      + C_w^{\tv} (\tau)
      \right),
  \end{align}
  while by~\ref{it:a} the function $a$ satisfies~\ref{it:H3}.
  By~\ref{it:v}, for every $i \in \N$ the function $c_i$
  satisfies~\ref{it:H1} and, moreover,
  $c_i (t) \in \C2 (\Omega; \R^n)$ for all $t \in [0,T]$ and
  $\nabla \div c_i \in \L1 ([0,T]\times\Omega; \R^n)$. In particular,
  thanks to~\ref{it:v} and~\eqref{eq:w_L1}, the following bounds hold
  for every $i \in \N$ and $t \in [0,T]$:
  \begin{align}
    \label{eq:v_divL1}
    \norm{\div c_i}_{\L1 ([0,t]; \L\infty (\Omega; \R))}
    \le \
    & K_v \norm{w_i}_{\L1 ([0,t] \times \Omega; \R)}
      \le \
      K_v \, t \, C_{w,1} (t),
    \\  \label{eq:v_DxL1}
    \norm{D_x c_i}_{\L1 ([0,t]; \L\infty (\Omega; \R^{n\times n}))}
    \le \
    &  K_v \norm{w_i}_{\L1 ([0,t] \times \Omega; \R)}
      \le \
      K_v \, t \,  C_{w,1} (t),
    \\ 
    \norm{\grad \div c_i (t)}_{\L1 (\Omega; \R^n)}
    \le \
    \label{eq:v_graddivL1}
    & C_v \left(t,  C_{w,1} (t)\right)  C_{w,1} (t).
  \end{align}
  The following uniform bounds on $u_i$ hold for every $i \in \N$:
  by~\Cref{lem:L1} and~\Cref{lem:tv}, exploiting
  also~\eqref{eq:w_L1}--\eqref{eq:w_TV}
  and~\eqref{eq:A_Linf}--\eqref{eq:v_graddivL1}, for all
  $\tau \in [0,T]$,
  \begin{align}
    \nonumber
    \norm{u_i (\tau)}_{\L1 (\Omega; \R)}
    \le \
    & \left(\norm{u_o}_{\L1 (\Omega; \R)}
      + \norm{a}_{\L1 ([0,\tau] \times \Omega; \R)}\right)
      \exp\left(K_\alpha \, \tau \left(1+C_{w,\infty} (\tau)\right)
      \right)
    \\
    \label{eq:u_L1}
    = \colon \
    & C_{u,1} (\tau),
    \\
    \nonumber
    \norm{u_i (\tau)}_{\L\infty (\Omega; \R)}
    \le \
    & \left(\norm{u_o}_{\L\infty (\Omega; \R)}
      + \norm{a}_{\L1 ([0,\tau]; \L\infty (\Omega; \R))}\right)
    \\ \nonumber
    & \qquad \times \exp\left(K_\alpha \, \tau \left(1+C_{w,\infty} (\tau)\right)
      + K_v \, \tau \, C_{w,1} (\tau)
      \right)
    \\
    \label{eq:u_Linf}
    = \colon \
    & C_{u,\infty} (\tau),
    \\
    \nonumber
    \tv\left(u_i (\tau, \cdot)\right)
    \le \
    &  \exp\left(
      K_\alpha \, \tau \left(1+C_{w,\infty} (\tau)\right)
      +
      K_v \, \tau \, C_{w,1} (\tau)
      \right)
    \\
    \nonumber
    & \times \Biggl(
      \tv(u_o) + \mathcal{O}(1) \norma{u_o}_{\L\infty(\Omega; \R)}
      + \int_0^\tau \tv\left( a (s)\right) \dd s \Biggr)
    \\ \nonumber
    & 
      + C_{u, \infty} (\tau)
      \Biggl(
      K_\alpha \tau \left(1
      + C_{w, \infty} (\tau)
      + C_w^{\tv} (\tau)\right)
      +
      C_v \, \tau \left(\tau,  C_{w,1} (\tau)\right)  C_{w,1} (\tau)
      \Biggr)
    \\
    \label{eq:u_TV}
    = \colon \
    & C_u^{\tv} (\tau).
  \end{align}
  By~\Cref{lem:A} and~\Cref{lem:c}, exploiting~\eqref{eq:A_Linf},
  \eqref{eq:u_L1} and~\eqref{eq:u_TV}, we get
  \begin{align}
    \nonumber
    \norma{u_{i+1} (t) - u_i (t)}_{\L1 (\Omega;\R)}
    \le \
    & 
      C_{u,1} (t)
      \int_0^t
      \norma{\div\left(c_i (\tau)  - c_{i-1} (\tau)\right)}_{\L\infty (\Omega;\R)} \dd\tau
    \\\nonumber
    & + C_u^{\tv} (t)
      \int_0^t \norma{c_i (\tau) -c_{i-1} (\tau)}_{\L\infty (\Omega; \R^n)} \dd\tau
    \\ \label{eq:35}
    & + \exp \left(
      t \,  K_\alpha \left(1+C_{w,\infty} (t)\right)
      \right)
    \\  \nonumber
    & \times
      \left(
      \norma{u_o}_{\L\infty (\Omega;\R)}
      +
      \norma{a}_{\L1 ([0,t]; \L\infty(\Omega;\R))}
      \right)
      \norma{A_i-A_{i-1}}_{\L1 ([0,t]\times\Omega;\R)} .
  \end{align}
  By~\eqref{eq:31}, exploiting the hypothesis~\ref{it:alpha} we obtain
  \begin{equation}
    \label{eq:34}
    \begin{aligned}
      \norma{A_i-A_{i-1}}_{\L1 ([0,t]\times\Omega;\R)} = \ & \int_0^t
      \int_\Omega \modulo{ \alpha\left(\tau,x,w_i (\tau,x)\right) -
        \alpha\left(\tau,x,w_{i-1} (\tau,x)\right) }\dd{x}\dd\tau
      \\
      \le\ & K_\alpha \, \norm{w_i - w_{i-1}}_{\L1
        ([0,t]\times\Omega;\R)}.
    \end{aligned}
  \end{equation}
  By~\eqref{eq:31}, exploiting the hypothesis~\ref{it:v}
  and~\eqref{eq:w_L1} we obtain
  \begin{align}
    \label{eq:39}
    \norma{\div\left(c_i (\tau)  - c_{i-1} (\tau)\right)}_{\L\infty (\Omega;\R)}
    \le \
    &
      C_v  \left( t , C_{w,1} (t)
      \right)
      \norm{w_i (\tau) - w_{i-1} (\tau)}_{\L1 (\Omega; \R)},
    \\
    \label{eq:40}
    \norma{c_i (\tau) -c_{i-1} (\tau)}_{\L\infty (\Omega; \R^n)}
    \le \
    & K_v  \norm{w_i (\tau) - w_{i-1} (\tau)}_{\L1 (\Omega; \R)}.
  \end{align}
  Hence, inserting~\eqref{eq:34}, \eqref{eq:39} and~\eqref{eq:40}
  into~\eqref{eq:35} yields
  \begin{align}
    \nonumber
    & \norma{u_{i+1} (t) - u_i (t)}_{\L1 (\Omega;\R)}
    \\
    \nonumber
    \le \
    & \Bigl(
      C_{u,1} (t) \,  C_v \! \left( t , C_{w,1} (t)\right)
      + C_u^{\tv} (t)
    \\ \label{eq:diff_ui}
    & +K_\alpha \, \exp \left(
      t \,  K_\alpha \left(1+C_{w,\infty} (t)\right)
      \right)
      \left(
      \norma{u_o}_{\L\infty (\Omega;\R)}
      +
      \norma{a}_{\L1 ([0,t]; \L\infty(\Omega;\R))}
      \right)
      \Bigr)
    \\
    \nonumber
    & \times
      \norm{w_i  - w_{i-1} }_{\L1 ([0,t] \times \Omega; \R)}.
  \end{align}

  \smallskip Collecting together~\eqref{eq:diff_wi}
  and~\eqref{eq:diff_ui} we obtain
  \begin{displaymath}
    \begin{aligned}
      & \norm{w_{i+1} - w_i}_{\L\infty([0,t];\L1 (\Omega; \R))} +
      \norm{u_{i+1} - u_i}_{\L\infty([0,t];\L1 (\Omega; \R))}
      \\
      \le \ & C_{u,w}(t) \, t \left( \norm{w_{i} - w_{i-i}}_{\L\infty
          ([0,t]; \L1 (\Omega; \R))} + \norm{u_i - u_{i-1}}_{\L\infty
          ([0,t]; \L1 (\Omega; \R))} \right),
    \end{aligned}
  \end{displaymath}
  where
  \begin{align*}
    C_{u,w}(t) = \
    &  K_\beta \,  e^{t \, K_\beta}
      C_{w, \infty} (t)
      +
      \Bigl(
      C_{u,1} (t) \,  C_v \! \left( t , C_{w,1} (t)\right)
      + C_u^{\tv} (t)
    \\
    & +K_\alpha \, \exp \left(
      t \,  K_\alpha \left(1+C_{w,\infty} (t)\right)
      \right)
      \left(
      \norma{u_o}_{\L\infty (\Omega;\R)}
      +
      \norma{a}_{\L1 ([0,t]; \L\infty(\Omega;\R))}
      \right)
      \Bigr).
  \end{align*}
  Choosing a sufficiently small $t_*>0$, we ensure that $(u_i,w_i)$ is
  a Cauchy sequence in the complete metric space
  $\L\infty \left([0,t_*]; \L1 (\Omega; \R^2)\right)$. Call $(u_*,w_*)$
  its limit.

  Then, the bounds~\eqref{eq:51} and~\eqref{eq:52} directly follow
  from~\eqref{eq:w_L1} and~\eqref{eq:w_Linf} by the lower
  semicontinuity of the $\L\infty$ norm with respect to the $\L1$
  distance. The same procedure applies to get~\eqref{eq:54}
  and~\eqref{eq:55} from~\eqref{eq:u_L1} and~\eqref{eq:u_Linf}. If
  $a \geq 0$, $b\geq 0$ and both components of the initial datum
  $(u_o, w_o)$ are positive, then also the components of $(u_*, w_*)$
  are positive.

  \smallskip

  We now prove that $(u_*,w_*)$ solves~\eqref{eq:1} in the sense
  of~\Cref{def:sol}. Note that by~\ref{it:v}, the sequence
  $v (\cdot, w_i)$ converges to $v (\cdot,w_*)$ in
  $\L\infty\left([0,t_*]; \L1 (\Omega;\reali)\right)$. Similarly,
  by~\ref{it:alpha} and~\ref{it:betaFinal}, $\alpha (\cdot,\cdot,w_i)$
  and $\beta (\cdot, \cdot, u_i, w_i)$ converge to
  $\alpha (\cdot,\cdot,w_*)$ and $\beta (\cdot, \cdot, u_*, w_*)$. Two
  applications of the Dominated Convergence Theorem ensure that the
  integral equality~\eqref{eq:41} for the hyperbolic problems
  and~\eqref{eq:4} for the parabolic problem do hold.

  By~\eqref{eq:u_Linf}, we also have
  $u_* \in \L\infty ([0,t_*]\times \Omega; \reali)$. Moreover,
  Lemma~\ref{lem:tcont} ensures that
  $u_* \in \C0 \left([0,t_*]; \L1(\Omega; \reali)\right)$, using
  also~\eqref{eq:A_Linf}--\eqref{eq:v_graddivL1}.

  By construction, we have
  $w_* \in \C0 \left([0,t_*]; \L1 (\Omega;\reali)\right)$. Indeed,
  the uniform bound~\eqref{eq:w_Linf} shows that
  $w_* \in \L\infty ([0,t_*]\times \Omega; \reali) \subseteq \L\infty
  \left([0,t_*]; \L1 (\Omega;\reali)\right)$. Moreover, a further
  application of the Dominated Convergence Theorem shows that $w_*$
  satisfies~\eqref{eq:5}. Hence, proceeding as in Claim~4 in the proof
  of~\Cref{prop:superStimePara}, we have that
  $w_* \in \C0 \left([0,t_*]; \L1 (\Omega;\reali)\right)$.

  Thus, $(u_*,w_*)$ satisfies the requirements in~\Cref{def:sol}.
  Moreover, this solution $(u_*, w_*)$ can be uniquely extended to all
  $[0,T]$. The proof is identical to~\cite[Theorem~2.2,
  Step~6]{SAPM2021}.

  \medskip

  Following the same techniques used in~\cite[Theorem~2.2,
  Step~7]{SAPM2021}, we can prove also the Lipschitz continuous
  dependence of the solution to~\eqref{eq:1} on the initial data. Let
  $(u_o, w_o)$ and $(\tilde u_o, \tilde w_o)$ be two sets of initial
  data. Call $(u,w)$ and $(\tilde u, \tilde w)$ the corresponding
  solutions to~\eqref{eq:1} in the sense of
  Definition~\ref{def:sol}. The proof is based on~\eqref{eq:11},
  \eqref{eq:B_Linf}, \eqref{eq:A_Linf}, \Cref{item:3} in~\Cref{lem:L1}
  and computations analogous to those leading to~\eqref{eq:diff_wi}
  and~\eqref{eq:diff_ui} now yield
  \begin{align*}
    &
      \norm{u(t)- \tilde u (t)}_{\L1 (\Omega; \R)}
      +
      \norm{w(t)- \tilde w (t)}_{\L1 (\Omega; \R)}
    \\
    \le \
    & \norm{u_o - \tilde u_o}_{\L1 (\Omega; \R)} \exp\left(K_\alpha \,t
      \left(1 + K_{w,\infty} (t)\right)\right)
      + \norm{w_o - \tilde w_o}_{\L1 (\Omega; \R)} e^{K_\beta \,t}
    \\
    & +  K_\beta \, e^{t \, K_\beta} K_{w, \infty} (t)
      \left( \int_0^t
      \norma{u (\tau) - \tilde u (\tau)}_{\L1 (\Omega;\R)}
      + \norma{w(\tau) - \tilde w (\tau)}_{\L1 (\Omega; \R)}
      \dd\tau
      \right)
    \\
    & + K_1(t)
      \int_0^t \norm{w (\tau)  -\tilde w (\tau)}_{\L1(\Omega; \R)}\dd\tau.
  \end{align*}
  where
  \begin{align}
    \nonumber
    K_{w, \infty} (t) = \
    & \min \left\{C_{w, \infty} (t),
      C_{\tilde w, \infty} (t)\right\},
    \\
    \label{eq:57}
    K_1 (t) = \
    & \min \left\{
      C_{u,1} (t) \,  C_v \! \left( t , C_{w,1} (t)\right)
      + C_u^{\tv} (t)  +K_\alpha \, C_{u,\infty} (t),
      \right.
    \\ \nonumber
    & \qquad\quad\left.
      C_{\tilde u,1} (t) \,  C_v \! \left( t , C_{\tilde w,1} (t)\right)
      + C_{\tilde u}^{\tv} (t)  +K_\alpha \, C_{\tilde u ,\infty} (t)
      \right\},
  \end{align}
  and $C_{\tilde w, 1}$, $C_{\tilde w, \infty}$, $C_{\tilde w}^{\tv}$,
  $C_{\tilde u, 1}$, $C_{\tilde u, \infty}$, $C_{\tilde u}^{\tv}$ are
  defined accordingly to~\eqref{eq:w_L1}, \eqref{eq:w_Linf},
  \eqref{eq:w_TV}, \eqref{eq:u_L1}, \eqref{eq:u_Linf},
  \eqref{eq:u_TV}, corresponding to the initial datum
  $(\tilde u_o, \tilde w_o)$.  Then, Gronwall
  Lemma~\cite[Lemma~3.1]{bressan-piccoli} yields
  \begin{align*}
    & \norm{u(t)- \tilde u (t)}_{\L1 (\Omega; \R)}
      +
      \norm{w(t)- \tilde w (t)}_{\L1 (\Omega; \R)}
    \\
    \le \
    &
      \left(
      \norm{u_o - \tilde u_o}_{\L1 (\Omega; \R)}
      + \norm{w_o - \tilde w_o}_{\L1 (\Omega; \R)}\right)
      \int_0^t \mathcal{K}_o (\tau)
      \exp\left(\int_\tau^t \mathcal{K} (s) \dd{s}\right) \dd\tau,
  \end{align*}
  with
  \begin{align*}
    \mathcal{K}_o (\tau) = \
    & \exp\left(
      \max\left\{
      \left(K_\alpha \, \tau \left(1 + K_{w,\infty} (\tau)\right)\right),
      K_\beta \, \tau
      \right\}
      \right),
    &
    \mathcal{K} (\tau) = \
    &  K_\beta \, e^{\tau \, K_\beta} K_{w, \infty} (\tau) + K_1 (\tau).
  \end{align*}
  Uniqueness of solution readily follows.

  \medskip

  Focus now on the stability of~\eqref{eq:1} with respect to the
  controls $a$ and $b$. Let $a, \tilde a$ satisfy~\ref{it:a},
  $b, \tilde b$ satisfy~\ref{it:b}. Call $(u,w)$ and
  $(\tilde u, \tilde w)$ the solutions to~\eqref{eq:1} corresponding
  to the functions $a,b$ and $\tilde a, \tilde b$ respectively.
  Similarly to the previous step, by~\eqref{eq:11}, \eqref{eq:B_Linf},
  \eqref{eq:A_Linf}, \Cref{item:3} in~\Cref{lem:L1} and computations
  analogous to those leading to~\eqref{eq:diff_wi}
  and~\eqref{eq:diff_ui}, we obtain
  \begin{align*}
    & \norm{u(t)- \tilde u (t)}_{\L1 (\Omega; \R)}
      +
      \norm{w(t)- \tilde w (t)}_{\L1 (\Omega; \R)}
    \\
    \le \
    & \norm{a-\tilde a}_{\L1 ([0,t]\times\Omega; \R)}
      \exp\left(K_\alpha \, t \left(1+ K_{w, \infty} (t)\right) \right)
      + \norm{b- \tilde b}_{\L1 ([0,t]\times\Omega; \R)} e^{K_\beta \, t}
    \\
    & +  K_\beta \, e^{t \, K_\beta} K_{w, \infty} (t)
      \left( \int_0^t
      \norma{u (\tau) - \tilde u (\tau)}_{\L1 (\Omega;\R)}
      + \norma{w(\tau) - \tilde w (\tau)}_{\L1 (\Omega; \R)}
      \dd\tau
      \right)
    \\
    & +  K_1(t)
      \int_0^t \norm{w (\tau)  -\tilde w (\tau)}_{\L1(\Omega; \R)}\dd\tau,
  \end{align*}
  where $K_{w, \infty} (t)$ and $K_1(t)$ are defined as
  in~\eqref{eq:57}, with the main difference that the
  \emph{tilde}-versions of $C_{*,*}$ now corresponds to the functions
  $\tilde a$ and $\tilde b$. An application of Gronwall
  Lemma~\cite[Lemma~3.1]{bressan-piccoli} yields the desired estimate.
\end{proofof}

\paragraph{Acknowledgement:}
The authors were partly supported by the GNAMPA~2022 project
\emph{Evolution Equations: Well Posedness, Control and
  Applications}.


{ \small

  \bibliography{dirichlet}

  \bibliographystyle{abbrv}

}

\end{document}